\documentclass[11pt]{amsart}

\usepackage{hyperref,wasysym}
\usepackage{placeins}
\hypersetup{colorlinks=true, citecolor=BurntOrange, linkcolor=ForestGreen}

\usepackage{amssymb}
\usepackage{graphicx}
\usepackage{amsmath}
\usepackage[usenames,dvipsnames]{xcolor}
\usepackage[text={6.5in,9in},centering]{geometry}
\newtheorem{Lemma}{Lemma}
\newtheorem{Proposition}[Lemma]{Proposition}
\newtheorem{Remark}[Lemma]{Remark}
\newtheorem{Theorem}[Lemma]{Theorem}

  {\begin{trivlist}\item[]\textbf{Proof#1 }}%
  {\hspace*{\fill}$\rule{.3\baselineskip}{.35\baselineskip}$\end{trivlist}}

\newcommand{\R}{\mathbb{R}}               % reals
\begin{document}

\title[Social outbursts in a tension inhibitive regime]{Traveling wave solutions in a model for social outbursts in a tension-inhibitive regime }

%\author[Ghazaryan]{
 
\date{\today}

\begin{abstract}
{ 
In this work we investigate the existence of non-monotone traveling wave solutions to a reaction-diffusion system 
modeling social outbursts, such as rioting activity, originally proposed in \cite{BJPN}.  The model consists of two scalar values, the 
level of unrest $u$ and a tension field $v$.  A key component of the model is a bandwagon effect in the unrest, provided the tension is sufficiently high.
We focus on  the so-called tension inhibitive regime, characterized by the fact that the level of unrest has a negative feedback on the tension. 
This regime has been shown to be physically relevant for the spatiotemporal spread of the 2005 French riots.
We use Geometric Singular Perturbation Theory to study the existence of such solutions in two situations.  The first is when both $u$ and $v$
diffuse at a very small rate.  Here, the time scale over which the bandwagon effect is observed plays a key role.  The second case we consider is when
the tension diffuses at a much slower rate than the level of unrest.  In this case, we are able to deduce that the driving dynamics are modeled by the well-known
Fisher-KPP equation. 
}
\end{abstract}
\maketitle

\begin{center}  Marzieh Bakhshi \textsuperscript{a},  Anna Ghazaryan \textsuperscript{a},  Vahagn Manukian \textsuperscript{a,b}, Nancy Rodriguez\textsuperscript{c}\\
 \end{center}
 
  \textsuperscript{a} \address{
  Department of Mathematics, Miami University, 301 S. Patterson Ave,   Oxford, OH 45056 USA,
  }  
 \email{bakhshm2@miamioh.edu, ghazarar@miamioh.edu }
 \par
% \textsuperscript{a} \address{ Department of Mathematics, Miami University, 301 S. Patterson Ave,   Oxford, OH 45056 USA,}  
%\email{ghazarar@miamioh.edu} \par
 \textsuperscript{b} \address{ Department of Mathematical and Physical Sciences, Miami University, 1601 University Blvd, Hamilton, OH 45011 USA,}  \email{manukive@miamioh.edu} \par
 \textsuperscript{c} \address{Department of Applied Mathematics,
Engineering Center, ECOT 225
526 UCB
Boulder, CO 80309-0526 USA,}  \email{rodrign@colorado.edu}


\textbf{AMS Classification:}
%97M70 %:Mathematical modeling, applications of mathematics: Behavioral and social sciences
%92D25, % Population Dynamics.
35Q91, %: PDEs in connection with game theory, economics, social and behavioral sciences
35B25, %(singular perturbation), 
%35B32, %(bifurcations),  
35K57, % Reaction-diffusion equations, 
35B36, %(Pattern formation), %
%35B40 %Asymptotic behavior of solutions,
34D15 %: ODEs: Singular perturbations

\section{Introduction}

Civil unrest, protests, and rioting are tools that populations use
to express objection or dissent towards an idea or  action, usually political.  These outbursts
of social activity have been ubiquitous in time and space and, in many cases, have changed the 
course of history.  From the religious protest in the early sixteenth century to the recent George Floyd protests \cite{Cheung}, 
which have engulfed the United States, these outbursts of activity amplify in time and have an underlying field of ``tension" driving them. 
In \cite{BJPN}, the authors introduce a reaction-diffusion model for the dynamics of rioting activity (or unrest) and
social tension, motivated by the 2005 French riots.  The model assumes a bandwagon effect on the level of unrest that turns on when the social tension is above a 
certain threshold
value.  Moreover, this model assumes a nearest-neighbor spread, in other words the spatial contagion is local and modeled by the classical diffusion operator. 
Some robust features observed in these social outbursts are the temporal up-and-down dynamics and, in cases like the 2005 French riots or the Velvet Revolution of 2018 in Armenia,
%\textcolor{blue}{\cite{VR},}
the spatial spread of the activity.  
These features have been observed in the data and 
can be expressed mathematically as the existence of traveling wave solutions.  

The system introduced in \cite{BJPN} has two regimes that are of 
interest, which can be represented by a parameter $p$.  The case when $p<0$ is known as the {\it tension enhancing} and the case when $p>0$ is known as the {\it tension inhibitive} \cite{BRR}.  These regimes are
characterized by the fact that the unrest has  a positive
or negative feedback on the social tension, respectively.   
The former case leads to a monotone system where classical techniques can provide significant insight into the
model dynamics, such as the existence and stability of traveling wave solutions \cite{VVV}. In this regime the traveling waves are monotone and thus 
do not present the temporal up-and-down dynamic feature observed in real life.  On the other hand, the {\it tension inhibitive} case loses monotonicity and is thus more challenging to
analyze.  However, this case does lead to the existence of non-monotone traveling wave solutions, which were explored
numerically in \cite{Yang}.  

In this work, we prove the existence of traveling wave solutions using Geometric Singular Perturbation theory  \cite{Jones94, kuehn}  in two 
sub-regimes of the tension inhibitive case.  We first consider the regime when the spatial spread of the level of unrest and the social tension
are small.  In this case, the parameter that sets the timescale over which the bandwagon effect would be observed, denoted by $\omega$, plays a key role in the analysis.  
Specifically, we consider the singular limits as $\omega\to 0$ and $\omega\to \infty$ to find the appropriate heteroclinic orbits.  We then use the theory of rotated vector fields \cite{Perko} for the
intermediate values of $\omega$.  We shall see in Section \ref{VDL} that in the limit as $\omega\to 0$ the dynamics of the system are driven by the dynamics of $u$ and evolve slowly along the $v$-nullcline, see Figure \ref{fig:fronts}.    Recall that the
time scale over which the bandwagon effect is observed is given by $\frac{1}{\omega}$, which goes to $\infty$ as $\omega\to 0$.  Thus, we expect that the dynamics of the level of unrest to dominate here.
On the other hand, as $\omega\to \infty$, the dynamics of the system are driven by the dynamics of $v$ and evolve slowly along the $u$-nullcline, see Figure \ref{fig:fronts2}.  Of course, here 
the time scale $\frac{1}{\omega}\to 0$ as $\omega\to \infty$ and the dynamics of the system are driven by the social tension.

The second case we consider is when the social tension diffuses at a much slower rate than the level of unrest.  Interestingly, the dynamics here
can be reduced to a Fisher-KPP type equation for the level of unrest.  
The case $p=0$ was analyzed in \cite{Yang} and decouples the dynamics between the level of unrest and social tension.  
In this case, the equation for the level of unrest also reduced to a Fisher-KPP equation with the social tension being equal to one.  The situation here is a bit different
as $v$ is a function of $u$, specifically $v=(1+u)^p$.  
Fisher-KPP equations have been found to model a wide range of biological phenomena, ranging from its original application in population genetics \cite{Fisher} to population dynamics in ecology \cite{Holmes} and 
wound healing \cite{Sherratt}.  Moreover, these type of equations are understood well from a mathematical point of view, see for example \cite{Ablowitz, Griffiths, KPP}.  Due to its ubiquity, the Fisher-KPP equation seems to be as fundamental to biology, ecology, and sociology, as the Navier-Stokes equation is to physics.   
A recent example that supports this is due to Berestycki, Roquejoffre, and Rossi (\cite{BRoqR}) who studied a classical epidemic SIR model with diffusion and with an additional compartment of infected individuals
traveling on a line with fast diffusion.  Interestingly, a classical transformation reduces the proposed model to a Fisher-KPP type equation.  This provides evidence that these seemingly 
different models, with very different source terms, are fundamentally related.  Our work provides additional evidence that the Fisher-KPP equation is fundamental in social applications.

{\it Outline:} We present the model and background information in Section \ref{background}.  In Section \ref{TW} we discuss the type of solutions that we seek and the 
model formulation that we use for each of the two cases to be considered.  In Section \ref{VDL} we discuss the vanishing diffusion limit case. In Section \ref{KPP} we
consider the reduction of the model of study to the Fisher-KPP equation and prove the existence of traveling wave solutions. We conclude with some numerical experiments 
in Section \ref{NR}.

\subsection{The model}\label{background}
Much research has led to the belief that certain external events are responsible for initiating a period of unrest \cite{Newburn2014},
the so-called {\it triggering events.}
However, 
one must also take into account long-established frustrations, which can play a role
in the intensity and duration of these social outbursts \cite{Lipsky1968a}.  This leads to a dynamic tension field, which is important to understand.  
The system proposed in \cite{BJPN} involves the coupling of an explicit variable representing the intensity of activity and an underlying tension field, as follows:%\vspace{-2pt}
%\begin{align}\label{sys:cont1}
%\left\{\begin{array}{l}
%u_t(x,t)= \Delta u(x,t) + r (v(x,t)) G(u(x,t)) - \omega u(x,t),\;x\in \R^n,\;t>0,\label{eq:u}\\ 
%v_t(x,t)= \Delta v(x,t) - h(u(x,t)) (x,t)+S(x,t),\;x\in \R^n,\;t>0,\label{eq:v}
%\end{array}\right.
%\end{align}
\begin{eqnarray}\label{e:01}
\left\{\begin{array}{cll}
u_{\tau}&=&d_1 \Delta u+r(v)u(1-u)-\omega u,\\
v_{\tau}&=&d_2\Delta v+1-h(u)v,
\end{array}\right. 
\end{eqnarray} 
satisfied for $\tau>0$ and $x\in \R^n$ and with non-negative initial data.  
The unknown $u$ represents the {\bf level of unrest} and $v$ measures the 
{\bf tension} in a system.
The function $G = u(1-u)$ is of  KPP-type \cite{Fisher} and models
self-excitement (or the so-called bandwagon effect).  This effect is assumed to be 
negligible until the tension $v$ is sufficiently large. This switch mechanism is described by the sigmoid-type function $r.$ 
The effect that $u$ has on $v$  
is modeled by the function $h(u):[0,\infty)\to(0,\infty),$
and is either monotone increasing or decreasing. 
The monotonicity of $h$ determines whether \eqref{e:01} is of
cooperative or activator-inhibitor type.  
For this reason, we refer to \eqref{e:01} in the 
case when $h$ is decreasing as 
a {\it tension enhancing system} and in the case when $h$
is increasing as a {\it tension inhibitive system}.
The specific functions considered are given by:
\begin{equation*}
r(v)=\frac{\Gamma}{1+e^{-\beta(v-\alpha)}}\quad\text{and} \quad h(u)=\theta (1+u)^p. 
\end{equation*}
The model also assumed a nearest neighbor contagion that is modeled by the diffusion terms $d_1\Delta u$ and $d_2 \Delta v$. 
Note that $p<0$ corresponds to the tension-enhancing case and $p>0$ to the tension-inhibitive case.  
Throughout the remainder of the paper we make the assumption that $\alpha=\theta=1$ and that 
$d_1$, $d_2$, $p$, $\Gamma$, $\beta$ are positive parameters.  in particular, we will be working in the
tension inhibitive case.

%\subsection{The model\label{RM}}

\section{Constant states and traveling wave solutions}\label{TW}

Our interest lies in studying planar traveling wave solutions and thus we can safely consider the one-dimensional version of \eqref{e:01}.  
To study the two distinct parameter regimes discussed above: (i) $d_1, d_2$ small and (ii) $d_2 \ll d_1,$ we view system \eqref{e:01} from different angles.  In the former case, we  
rename $\Gamma/\omega=\gamma$ and recast \eqref{e:01} as:
\begin{eqnarray}\label{e:02stationary}
\left\{\begin{array}{cll}
u_{\tau}&=&d_1 u_{xx}+\omega\left(\frac{\gamma}{1+e^{-\beta(v-1)}}u(1-u)- u\right),\\
v_{\tau}&=&d_2v_{xx}+1- (1+u)^pv.
\end{array}\right.
\end{eqnarray} 
For the latter case,   with abuse of notation, we  replace the time variable $\tau$ with  $ \omega \tau$ and  spatial variable $x$ with $\sqrt{\omega} x$ and get an equivalent system:
\begin{eqnarray}\label{e:02stationary_2}
\left\{\begin{array}{cll}
u_{\tau}&=&d_1 u_{xx}+\frac{\gamma}{1+e^{-\beta(v-1)}}u(1-u)- u,\\
v_{\tau}&=&d_2v_{xx}+\frac{1}{\omega}\left(1- (1+u)^pv\right).
\end{array}\right. 
\end{eqnarray} 
%The systems \eqref{e:02stationary}-\eqref{e:02stationary_2} will be used consequently to better understand two separate situations, thus, using the same variables does not create a conflict.

To find the constant states of \eqref{e:01} (equivalently of \eqref{e:02stationary} and \eqref{e:02stationary_2}), we solve the  system of algebraic equations: 
\begin{equation*}
\frac{\gamma}{1+e^{-\beta(v-1)}}u(1-u)- u=0, \quad 1- (1+u)^pv=0.
\end{equation*}
As illustrated in Fig.~ \ref{fig:fronts1}, there are two physically relevant constant states: $A(0,1)$ and $B (\bar{u},\bar{v})$, where $\bar{u},\bar{v}>0$.  More precisely, 
 %$\bar{v} = \frac{1}{(1+u)^p}$ and 
 $\bar{u}$ is defined as the solution of the transcendental equation:
\begin{equation}\gamma -1 - \gamma u =  e^{-\beta( \frac{1}{(1+u)^p}-1)}\label{baru}\end{equation}
and then
\begin{equation}\bar{v} = \frac{1}{(1+\bar u)^p}.\label{barv}\end{equation}
The constant state $A(0,1)$ is the {\it relaxed} state with no activity and $B(\bar{u},\bar{v})$ is the {\it excited} state with a positive level of activity.  
%\textcolor{blue}{Should we include a graph of the null-clines here?}
\begin{figure}[t]
\centering\begin{minipage}[c]{0.38\textwidth}
\scalebox{0.16}{ \includegraphics{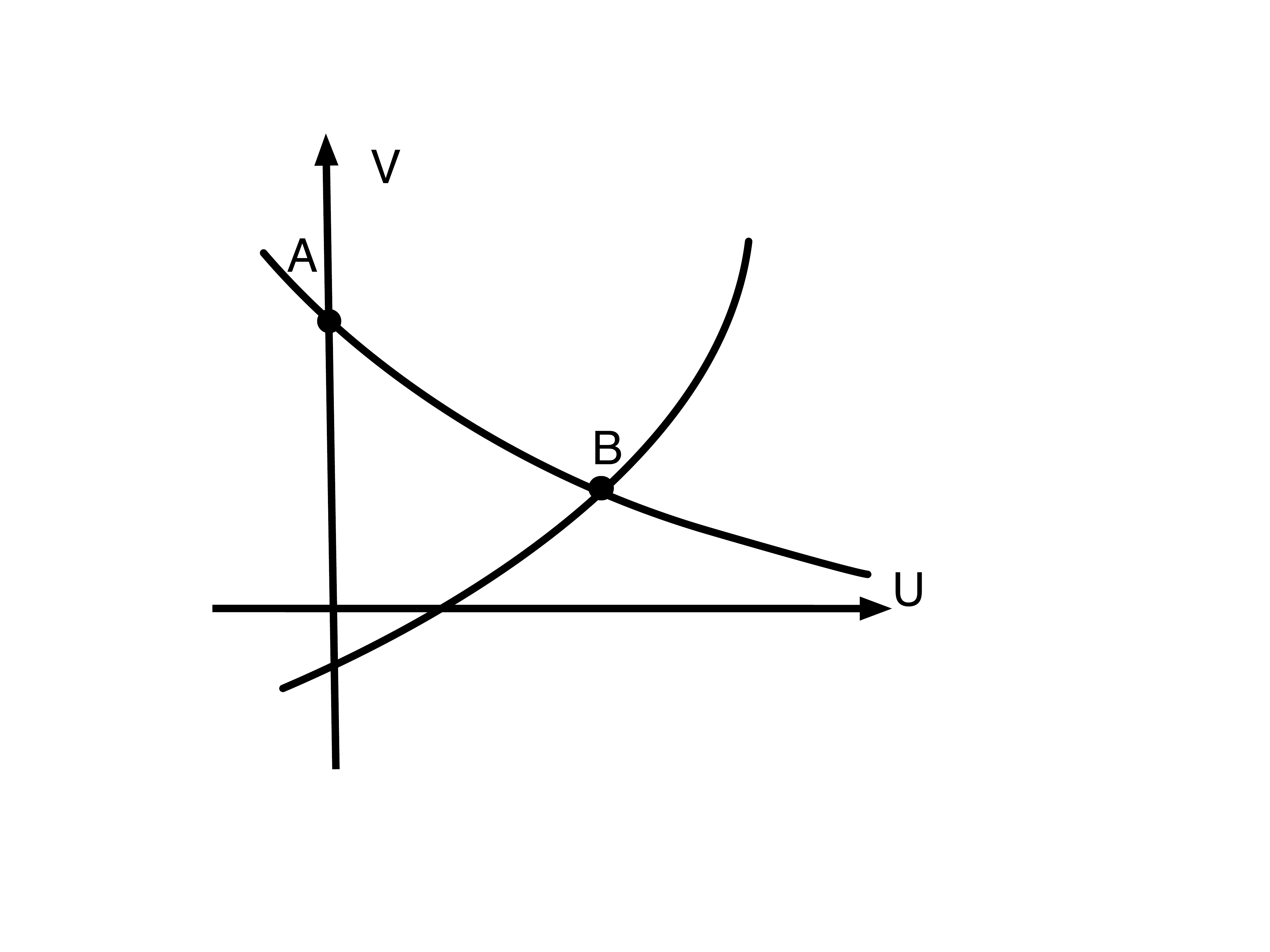}}\end{minipage}
\hspace{0.4in}
\begin{minipage}[c]{0.38\textwidth}
\scalebox{0.16}{ \includegraphics{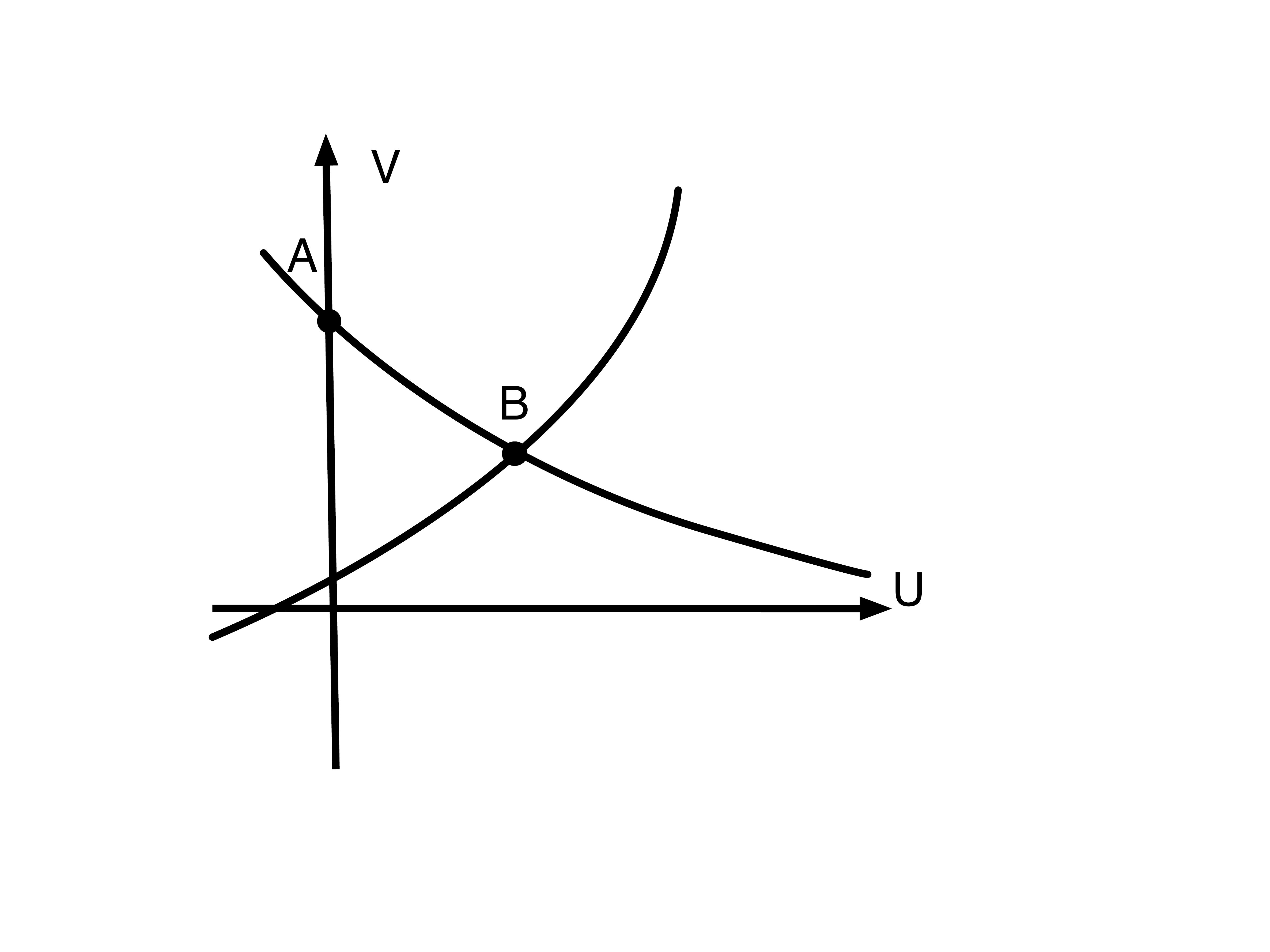}}\end{minipage}
%\begin{center}$
%\begin{array}{cccc}
%\includegraphics[width=2.4in]{fig1} &
%\includegraphics[width=2.4in]{fig2} 
%\end{array}$
\caption{The constant states of system \eqref{e:01}  are  the points of the intersection of the nullclines illustrated  for $p>0$, $\gamma>2$ with: (left panel)  $\gamma-1-e^\beta>0$  and  (right panel) $\gamma-1-e^\beta<0$ . 
%\textcolor{blue}{no a) or b) in the subfigures}
}
\label{fig:fronts1}
%\end{center}
\end{figure}
To study traveling wave solutions, it is convenient to introduce a moving coordinate frame $\xi=x-c\tau$, where $c$ is the propagating speed of the front.  Note that due to the symmetry $(c,\xi) \leftrightarrow(-c, -\xi)$, it is enough to 
consider $c>0.$ In
the new variable $\xi=x-c\tau$, the system given by \eqref{e:01} reads as follows:
\begin{eqnarray}\label{e:02a}
\left\{\begin{array}{l}
u_{\tau}=d_1 u_{\xi\xi}+cu_\xi+r(v)u(1-u)-\omega u,\\
v_{\tau}=d_2v_{\xi\xi}+cv_\xi+1-h(u)v.
\end{array}\right.
\end{eqnarray} 

Traveling wave solutions do not change their profile in time, so the corresponding traveling wave ODE system to \eqref{e:01} is given by:
\begin{eqnarray}\label{e:tw}
\left\{\begin{array}{l}
0=d_1 u_{\xi\xi}+cu_\xi++r(v)u(1-u)-\omega u,\quad\text{for}\; \xi \in \R,\\
0=d_2v_{\xi\xi}+cv_\xi+1-h(u)v,\hspace{62pt}\text{for}\;\xi \in \R,\\
(u(-\infty), v(-\infty)) = B (\bar{u},\bar{v})\quad\text{and}\quad (u(\infty), v(\infty)) = A(0,1),
\end{array}\right.
\end{eqnarray} 
where we have used the notation $u(\pm\infty)=\lim_{\xi\to\pm\infty} u(x)$ and $v(\pm\infty)=\lim_{\xi\to\pm\infty} v(x)$.

%We are planning to focus on two situations for \eqref{e:01}.
%The first situation we describe as  a small diffusion case. We assume that both $d_1$ and $d_2$ are  very small and consider  \eqref{e:03} as a singular perturbation of a related vanishing diffusion limit.  In Section~\ref{VDL} we use  Geometric Singular Perturbation theory to prove existence of traveling waves in that case.
%
%A different type of fronts are recovered by the assumption that the quantity $v$ diffuses at a rate which much smaller than the diffusion rate of $u$ and other parameters in the system \eqref{e:01}. We investigate  this regime  and  prove existence of traveling waves in Section!\ref{KPP}.
%
%We conclude the paper with Section~\ref{NR} where representative front solutions are obtained numerically.

\section{Vanishing diffusion limit \label{VDL}}
In this section, we consider the case when $d_1,d_2\ll 1$.  
%We are planning to focus on two situations for \eqref{e:03}.
%The first situation we describe as  a small diffusion case. 
Here we study the traveling wave ODE system corresponding to \eqref{e:02stationary}, which reads as: 
\begin{eqnarray}\label{e:03}
\left\{\begin{array}{l}
0=d_1 u_{\xi\xi}+cu_\xi+\omega\left(\frac{\gamma}{1+e^{-\beta(v-1)}}u(1-u)- u\right),\\
0=d_2v_{\xi\xi}+cv_\xi+1- (1+u)^pv.
\end{array}\right.
\end{eqnarray} 
We will consider  \eqref{e:03} as a singular perturbation of a related vanishing diffusion limit. 
To reflect that both diffusion coefficients  $d_1$  and $d_2$ are small and comparable parameters,  we introduce the following notation:  \begin{equation} d_1=\epsilon, \mbox{ where }   0< \epsilon \ll 1  \mbox{ and }  d_2=\mu d_1,  \mbox{ where } 0<\mu=O(1).\end{equation}   
The corresponding version of \eqref{e:02a} and \eqref{e:03} are as follows: 
\begin{eqnarray}\label{e:02pde}
\left\{\begin{array}{l}
u_{\tau}=\epsilon u_{\xi\xi}+cu_\xi+\omega\left(\frac{\gamma}{1+e^{-\beta(v-1)}}u(1-u)- u\right),\\
v_{\tau}=\epsilon \mu v_{\xi\xi}+cv_\xi+1- (1+u)^pv.
\end{array}\right.
\end{eqnarray} 
and 
\begin{eqnarray}\label{e:04}
\left\{\begin{array}{l}
0=\epsilon u_{\xi\xi}+cu_\xi+\omega\left(\frac{\gamma}{1+e^{-\beta(v-1)}}u(1-u)- u\right),\\
0=\epsilon \mu v_{\xi\xi}+cv_\xi+1- (1+u)^pv.
\end{array}\right.
\end{eqnarray} 
%In this section we prove the following theorem.
%\begin{Theorem} Assume    that $\gamma>2$. For any fixed values of $\mu>0$, there exist $\epsilon_0>0$ such that for  any $\epsilon <\epsilon_0$ in \eqref{e:02pde} there exist a translationally invariant  family of fronts that have the  equilibria $A=(0,1)$ and $B=(\bar{u},\bar{v})$ as rest states. \end{Theorem}
To prove the existence of a traveling wave solution $(u,v,c)$ which satisfies \eqref{e:tw},
we use Applied Dynamical Systems techniques. More precisely, when $\epsilon \ll 1$
the dynamical system associated to the ODE system \eqref{e:04}   is a singular perturbation of a lower-dimensional  dynamical system, therefore it is natural to use Geometric Singular Perturbation theory.  
We seek traveling fronts  of equation \eqref{e:02pde}  as heteroclinic  orbits for the  first-order system: 
\begin{eqnarray}\label{e:05s}
\left\{\begin{array}{lll}
\frac{du_1}{d\xi}&=& u_2,\\
\epsilon\frac{du_2}{d\xi}&=&-c u_2-\omega\left(\frac{\gamma}{1+e^{-\beta(v_1-1)}}u_1(1-u_1)- u_1\right),\\
\frac{dv_1}{d\xi}&=& v_2,\\
\epsilon   \mu  \frac{dv_2}{d\xi}&=&-cv_2- 1+ (1+u_1)^pv_1.
\end{array}\right.
\end{eqnarray} 
We call system \eqref{e:05s} a slow system, as opposed to the fast system that is obtained from \eqref{e:05s} through the scaling $\zeta=\xi/\epsilon$: 
\begin{eqnarray}\label{e:05f}
\left\{\begin{array}{lll}
\frac{du_1}{d\zeta}&=&\epsilon u_2,\\
\frac{du_2}{d\zeta}&=&-c u_2-\omega\left(\frac{\gamma}{1+e^{-\beta(v_1-1)}}u_1(1-u_1)- u_1\right),\\
\frac{dv_1}{d\zeta}&=& \epsilon v_2,\\
\mu \frac{dv_2}{d\zeta}&=&-c v_2- 1+ (1+u_1)^pv_1.
\end{array}\right.
\end{eqnarray} 

We  next consider the limit of the systems \eqref{e:05s} and \eqref{e:05f} as  $\epsilon\to 0$. Since  $\mu=O(1)$, then $\mu \epsilon \to 0$ as well.  In this limit, system \eqref{e:05s} produces a description of the set  that the solution belongs to
\begin{equation} \label{e:09}
M_0=\left\{(u_1,u_2,v_1,v_2)|u_2=-\frac{\omega}{c}\left(\frac{\gamma}{1+e^{-\beta(v_1-1)}}u_1(1-u_1)- u_1\right),  v_2= \frac{1}{c}(-1+(1+u_1)^pv_1)\right\}.
\end{equation}
On $M_0,$  the dynamics of the slow variables $u_1$ and $v_1$  are given  by: 
\begin{eqnarray}\label{e:10_0}
\left\{\begin{array}{lll}
\frac{du_1}{d\xi}&=&-\frac{\omega}{c}\left(\frac{\gamma}{1+e^{-\beta(v_1-1)}}u_1(1-u_1)- u_1\right),\\
\frac{dv_1}{d\xi}&=& -\frac{1}{c}(1-(1+u_1)^pv_1). 
\end{array}\right.
\end{eqnarray} 
The  set $M_0$ %, defined in \eqref{e:09}, 
also serves as a set of equilibrium points for \eqref{e:05f} with $\epsilon=0$, 
\begin{eqnarray}\label{e:11}
\left\{\begin{array}{lll}
\frac{du_1}{d\zeta}&=&0,\\
\frac{du_2}{d\zeta}&=&-c u_2-\omega\left(\frac{\gamma}{1+e^{-\beta(v_1-1)}}u_1(1-u_1)- u_1\right),\\
\frac{dv_1}{d\zeta}&=& 0,\\
\mu \frac{dv_2}{d\zeta}&=&-c v_2- 1+ (1+u_1)^pv_1.
\end{array}\right.
\end{eqnarray} 
The linearization of  \eqref{e:11} about any point of the set $M_0$, defined in \eqref{e:09}, has two zero eigenvalues and two eigenvalues  equal to $-c$.
Therefore,  the set $M_0$ is a normally hyperbolic  and an  attracting set. By the Fenichel's invariant manifold theory \cite{Fenichel79, Jones94} 
there exists an $\epsilon$-order perturbation of  $M_0,$ which is an invariant manifold for \eqref{e:05s}, equivalently for \eqref{e:05f}:
\begin{equation} \label{e:09eps}
M_{\epsilon}=\left\{(u_1,u_2,v_1,v_2)|u_2=-\frac{\omega}{c}\left(\frac{\gamma u_1(1-u_1)}{1+e^{-\beta(v_1-1)}}- u_1\right) +O(\epsilon),  v_2= \frac{1}{c}(-1+(1+u_1)^pv_1)+O(\epsilon) \right\}.
\end{equation}
On that manifold the flow  generated by \eqref{e:05s} is then an $\epsilon$-order perturbation  of the flow \eqref{e:10_0},
\begin{eqnarray}\label{e:10_0e}
\left\{\begin{array}{l}
\frac{du_1}{d\xi}=-\frac{\omega}{c}\left(\frac{\gamma}{1+e^{-\beta(v_1-1)}}u_1(1-u_1)- u_1\right)+O(\epsilon),\\
\frac{dv_1}{d\xi}= -\frac{1}{c}(1-(1+u_1)^pv_1)+O(\epsilon), 
\end{array}\right.
\end{eqnarray} 
so the slow dynamics of \eqref{e:05s} is restricted to the two-dimensional set \eqref{e:09eps}.
%The reduced system on the normally attracting invariant manifold 
%\begin{equation*}
%M_0=\{(u_1,u_2,v_1,v_2)|u_2=-\frac{\omega}{c}\left(\frac{\gamma}{1+e^{-\beta(v_1-\alpha)}}u_1(1-u_1)- u_1\right),  v_2= \frac{\mu}{c}(-1+(1+u_1)^pv_1)\}\qquad 
%\end{equation*}
%is the following system equations
%\begin{eqnarray}\label{e:31_0f}
%\frac{du_1}{d\xi}&=&-\frac{\omega}{c}\left(\frac{\gamma}{1+e^{-\beta(v_1-\alpha)}}u_1(1-u_1)- u_1\right)\notag\\
%\frac{dv_1}{d\xi}&=& -\frac{\mu}{c}(1-(1+u_1)^pv_1)
%\end{eqnarray} 
%\textcolor{red}{ WHY? We denote $z = \xi/c $ and rewrite  the reduced system  \eqref{e:10_0} as}
%\begin{eqnarray}\label{e:10}
%\frac{du_1}{d z}&=&-\omega u_1\left(\frac{\gamma}{1+e^{-\beta(v_1-1)}}(1-u_1)- 1\right)\notag\\
%\frac{dv_1}{d z}&=& -1+(1+u_1)^pv_1.
%\end{eqnarray} 
The nullclines of the planar system \eqref{e:10_0} are given by: 
\begin{equation}\label{nullclines}
 u_1=1-\frac{1}{\gamma}(1+e^{-\beta(v_1-1)}), \quad u_1=0, \quad v_1=\frac{1}{(u_1+1)^p}.
\end{equation}
Note that there are no equilibrium solutions in the open first quadrant when $\gamma\le 2$, therefore we will only consider the case when $\gamma>2$.
When $\gamma>2$, there are two relevant equilibria:
$A=(0,1)$ and $B=(\bar{u},\bar{v})$, where the components of $B$ are described in \eqref{baru}-{\eqref{barv}. 
%$\bar{u}$ is defined as the solution of the transcendental equation $\gamma(1-u) = 1+ e^{-\beta\left( \frac{1}{(1+u)^p}-1\right)}$ and  $\bar{v} = \frac{1}{(1+\bar{u})^p}$.
In \cite[Theorem 2.1]{Yang} it is proved that  in the system:
\begin{eqnarray}\label{e:00}
\left\{\begin{array}{l}
u_{t}= \frac{\gamma}{1+e^{-\beta(v-1)}} u(1-u)- u,\\
v_{t}= 1-(1+u)^pv
\end{array}\right.
\end{eqnarray}
the non-trivial steady state $(\bar{u},\bar{v})$ with positive components is globally stable in the open first quadrant.  The system \eqref{e:00} is a scaled version of \eqref{e:10_0} with reversed dynamics.  
The global stability of $(\bar{u},\bar{v})$  in \eqref{e:00} implies global stability of the corresponding equilibrium $(\bar{u},\bar{v})$ in \eqref{e:10_0}  in reversed ``time"  $\xi$.

The linearization of the vector field generated by  \eqref{e:10_0} at the equilibrium $A$ has the  eigenvalues $\lambda_1=1$ and $\lambda_2=\frac{2-\gamma}{2} \omega$, so  $A$ is a saddle when $\gamma>2$ and it is a node when $\gamma<2$.   
The global stability of $B=(\bar{u},\bar{v})$ in reversed ``time"  $\xi$ implies that for $\gamma >2$ the equilibria $A$ 
and $B$ are connected along the stable manifold of the saddle $A$. We give a detailed geometric description of the structure of this orbit below.

For brevity we introduce the following notation: 
\begin{eqnarray}
\left\{\begin{array}{l}
f_1(u_1,v_1)=-\left(\frac{\gamma}{1+e^{-\beta(v_1-1)}}u_1(1-u_1)- u_1\right),\\
 f_2(u_1,v_1)=-(1-(1+u_1)^pv_1),
\end{array}\right.
\end{eqnarray} 
thus \eqref{e:10_0} now reads as follows: 
\begin{eqnarray}
\left\{\begin{array}{l}
\frac{du_1}{d z}= \frac{\omega}{c}f_1(u_1,v_1),\\
\frac{dv_1}{d z}=  \frac{1}{c} f_2(u_1,v_1).
\end{array}\right.
\end{eqnarray}
%\textcolor{red}{ is $\omega$ part of $f_1$?}
The eigenvalues of the  linearization  of \eqref{e:10_0} at $B$ are as follows:
%\begin{equation}\label{evw} \lambda_{1,2}(B)=\frac{1}{2c}\left( f_{2v_1}(B)+\omega f_{1u_1}(B)\pm\sqrt{( f_{2v_1}(B)+\omega f_{1u_1}(B))^2-4\omega(f_{1u_1}(B)f_{2v_1}(B)-f_{1v_1}(B)f_{2u_1}(B))}\right) \end{equation}
%or
\begin{equation}\label{evw1}
\lambda_{1,2}(B)=\frac{1}{2c}\left( f_{2v_1}(B)+\omega f_{1u_1}(B)\pm\sqrt{( f_{2v_1}(B)-\omega f_{1u_1}(B))^2+4\omega f_{1v_1}(B)f_{2u_1}(B)}\right)
\end{equation}
%\begin{figure}[t]
%\begin{center}$
%\begin{array}{cccc}
%\includegraphics[width=3in]{fig1} &
%\includegraphics[width=3in]{fig2} 
%\end{array}$
%\caption{ Nullclines when $p>0$, $\gamma>2$ a) $\gamma-1-e^\beta>0$;  b) $\gamma-1-e^\beta<0$}
%\label{fig:fronts}
%\end{center}
%\end{figure}
where
\begin{eqnarray}
\left\{\begin{array}{cll}
f_{1u_1}(B)&=&\frac{\gamma-1-e^{-\beta(\bar{v}_1-1)}}{1+e^{-\beta(\bar{v}_1-1)}} = \frac{\bar{u}_1}{1-\bar{u}_1},\\
f_{1v_1}(B)&=&-\frac{ \bar{u}_1\beta e^{-\beta(\bar{v}_1-1)}}{1+e^{-\beta(\bar{v}_1-1)}} 
%= -\frac{\beta \bar{u}_1 (\gamma(1-\bar{u}_1)-1)}{1-\bar{u}_1}
= -\beta \bar{u}_1 \left(\gamma -\frac{1}{1-\bar{u}_1}\right),\\
f_{2u_1}(B)&= &p(1+\bar{u}_1)^{p-1}\bar{v}_1 =  \frac{p}{1+\bar{u}_1},\\
f_{2v_1}(B)&=&(1+\bar{u}_1)^p.
\end{array}\right.
\end{eqnarray}
Since $\gamma>2$ and  $\bar{v}_1>0$, it is easy to see  that $f_{1u_1}(B)>0$,  $f_{2u_1}(B)>0$, $f_{2v_1}(B)>0$ and $f_{1v_1}(B)<0$, and so the equilibrium $B$ is an unstable node. 
The eigenvalues \eqref{evw1} may be real or complex depending on the parameters of the system.
Note that 
\begin{equation}
f_{2v_1}(B)+\omega f_{1u_1}(B)>0
\end{equation}
and the expression under the root sign in \eqref{evw1} becomes zero at the points:
\begin{equation}
\begin{array}{lll}
\omega_{1}&=&\frac{f_{1u}f_{2v} - 2f_{1v}f_{2u} -2\sqrt{-f_{1u}f_{1v}f_{2u}f_{2v} + f_{1v}^2f_{2u}^2}}{f_{1u}^2},\\
\omega_{2}&=&\frac{f_{1u}f_{2v} - 2f_{1v}f_{2u} + 2\sqrt{-f_{1u}f_{1v}f_{2u}f_{2v} + f_{1v}^2f_{2u}^2}}{f_{1u}^2}.
\end{array}
\end{equation}

From $f_{1v_1}(B)<0$ it follows  that $\omega_2>0$ and $\omega_2>\omega_1$. Since for small $\omega$ both  eigenvalues $\lambda_{1,2}(B)$ are positive, then $\omega_1>0$. Therefore, $\lambda_{1,2}(B)$ are:
\begin{itemize}
\item positive for $\omega\in(0,\omega_1)\cup (\omega_1,\infty);$
\item complex with positive real part for $\omega\in(\omega_1,\omega_2)$.
\end{itemize}

To analyze the dynamics of  the system \eqref{e:10_0} we consider separately  the  cases when $\omega \ll1$ and  $\omega \gg1$, and then discuss the situation of the intermediate values of  $\omega$.
In the first case, when $\omega \ll1$ the following theorem holds. 

\begin{Theorem} Assume    that $\gamma>2$,  $\mu>0$ and $c>0$ are fixed parameters. Assume also that $0<\epsilon \ll \omega$. There exists $\omega_0>0$ such that for any $0<\omega<\omega_0$, there is $\epsilon_0 =\epsilon(\omega) >0$ such that for  any $\epsilon <\epsilon_0$ in the system \eqref{e:05s}. Equivalently, for the system \eqref{e:05f}, there exists a heteroclinic  orbit connecting $(0,0,1,0)$ and $(\bar{u},0,\bar{v},0).$ Thus, for 
 \eqref{e:02pde} there exists a translationally invariant  family of fronts that have the constant states $A=(0,1)$ and $B=(\bar{u},\bar{v})$ as their rest states. \end{Theorem}
\begin{proof}
Let us consider system \eqref{e:10_0}
 %\begin{eqnarray*} \frac{du_1}{d\xi}&=&-\frac{\omega}{c}\left(\frac{\gamma}{1+e^{-\beta(v_1-1)}}u_1(1-u_1)- u_1\right),\\ \frac{dv_1}{d\xi}&=& -\frac{1}{c}(1-(1+u_1)^pv_1). \end{eqnarray*}  
along with a rescaled version of it, in terms of the variable $\eta = \omega \xi$, 
\begin{eqnarray}\label{e:10eta}
\left\{\begin{array}{cll}
\frac{du_1}{d\eta}&=&-\frac{1}{c}\left(\frac{\gamma}{1+e^{-\beta(v_1-1)}}u_1(1-u_1)- u_1\right),\\
\omega \frac{dv_1}{d\eta} &=& -\frac{1}{c}(1-(1+u_1)^pv_1). 
\end{array}\right.
\end{eqnarray}
When  $\omega =0$, the system   \eqref{e:10_0} becomes:
%\textcolor{red}{ $\omega=0$:} 
\begin{eqnarray}\label{lin}
\left\{\begin{array}{cll}
\frac{du_1}{d \xi}&=&0,\\
\frac{dv_1}{d \xi}& = & -\frac{1}{c}\left(1-(1+u_1)^pv_1\right).
\end{array}\right.
\end{eqnarray}
On the other hand, when we set  $\omega =0$ in \eqref{e:10eta} %we obtain the system:
\begin{eqnarray}\label{e:10eta0}
\left\{\begin{array}{cll}
\frac{du_1}{d\eta}&=&-\frac{1}{c}\left(\frac{\gamma}{1+e^{-\beta(v_1-1)}}u_1(1-u_1)- u_1\right),\\
0 &=& -\frac{1}{c}(1-(1+u_1)^pv_1), 
\end{array}\right.
\end{eqnarray}
 we obtain the manifold to which the solution of this reduced system belongs:
\begin{equation}\label{v1}\left\{ (u_1,v_1):\,v_1=\frac{1}{(1+u_1)^p}\right\}\end{equation}
and the reduced flow on this manifold: 
\begin{equation}\label{u10}
\frac{du_1}{d \eta}=- \frac{1}{c}u_1 \left(\frac{\gamma}{1+e^{-\beta({\frac{1}{(1+u_1)^p}}-1)}}(1-u_1)- 1\right).
\end{equation}
%which is obtained from  \eqref{e:10eta0}.
% The latter equation was obtained  by rescaling \eqref{e:10} in terms of $\eta$ and the setting $\omega=0$. 
Equation \eqref{u10} has two equillibrium points: $\tilde{A}=0$ and $\tilde{B}=\bar{u}_1$. It is easy to see
 %from \eqref{evw} 
 that the linearization of \eqref{lin} about  any point $(u_1,v_1)$ of \eqref{v1}  has a positive eigenvalue $p(1+u_1)^{p-1}/c$ and a zero eigenvalue, so the set \eqref{v1} is normally hyperbolic and repelling. 
  
  The linearization of \eqref{u10} about   $\tilde{A}$  has a negative eigenvalue, while the linearization of \eqref{u10} about   $u_1=\bar{u}_1$ has  a  positive eigenvalue, so      $u_1=0= \tilde{A}$ is a stable node and  $\tilde{B}$ is an unstable node. 
 Therefore, there is an asymptotic connection from $\tilde{B}$ at $-\infty$ to $\tilde{A}$ at $\infty$. Within the one-dimensional slow manifold \eqref{v1}, 
  %$v_1=\frac{1}{(1+u_1)^p}$, 
 this intersection is transversal  by the dimension counting. 
 Since the slow manifold  \eqref{v1} is normally hyperbolic, by Fenichel's invariant manifold theory \cite{Fenichel79, Jones94}  it persists when a sufficiently small $\omega$ is introduced, {\it i.e.}, there is an invariant manifold in \eqref{e:10_0} which is also normally repelling and is an $\omega$-order perturbation of \eqref{e:10_0}:
 \begin{equation}\label{v1omega}v_1=\frac{1}{(1+u_1)^p}+O(\omega),\end{equation}
 on which the flow is an $\omega$-order perturbation  of \eqref{u10} given by:
\begin{equation}\label{u10omega}
\frac{du_1}{d \eta}=- \frac{1}{c}u_1 \left(\frac{\gamma}{1+e^{-\beta({\frac{1}{(1+u_1)^p}}-1)}}(1-u_1)- 1\right) +O(\omega).
\end{equation}
Since the set is repelling, the stable manifold of the saddle  $(0,1)$ must stay on the manifold.  This stable manifold  then intersects with the  the unstable manifold of the equilibrium $( \bar u, \bar v)$;  thus, forming a heteroclinic orbit along the  set \eqref{v1omega}. 
In the two-dimensional phase space, the intersection of the one-dimensional stable manifold of the saddle  $(0,1)$ with the two-dimensional  unstable manifold of the node $(\bar u, \bar v)$ is transversal by the dimension counting. 

This geometric construction of a heteroclinic orbit  is performed on the slow manifold $M_0$ of the system  \eqref{e:11}, which was shown above to be normally hyperbolic and attracting. For a sufficiently small $\epsilon>0$,  the slow manifold $M_0$ perturbs to an attracting, two-dimensional invariant set $M_{\epsilon}$. Since $M_{0}$ is attracting, the  two-dimensional unstable manifold of equilibrium  $(\bar u, 0, \bar v, 0)$  is confined to $M_{0}$, and thus any orbit that follows this manifold is also confined to $M_{0}$. Therefore, within $M_{\epsilon}$, the intersection of two-dimensional unstable manifold of equilibrium  $(\bar u, 0, \bar v, 0)$   and   the one-dimensional {\it slow} stable manifold of the equilibrium $(0,0,1,0)$  persists,  forming a ``slow" heteroclinic orbit.

\begin{figure}[t]
\begin{center}$
\begin{array}{cccc}
\includegraphics[width=3in]{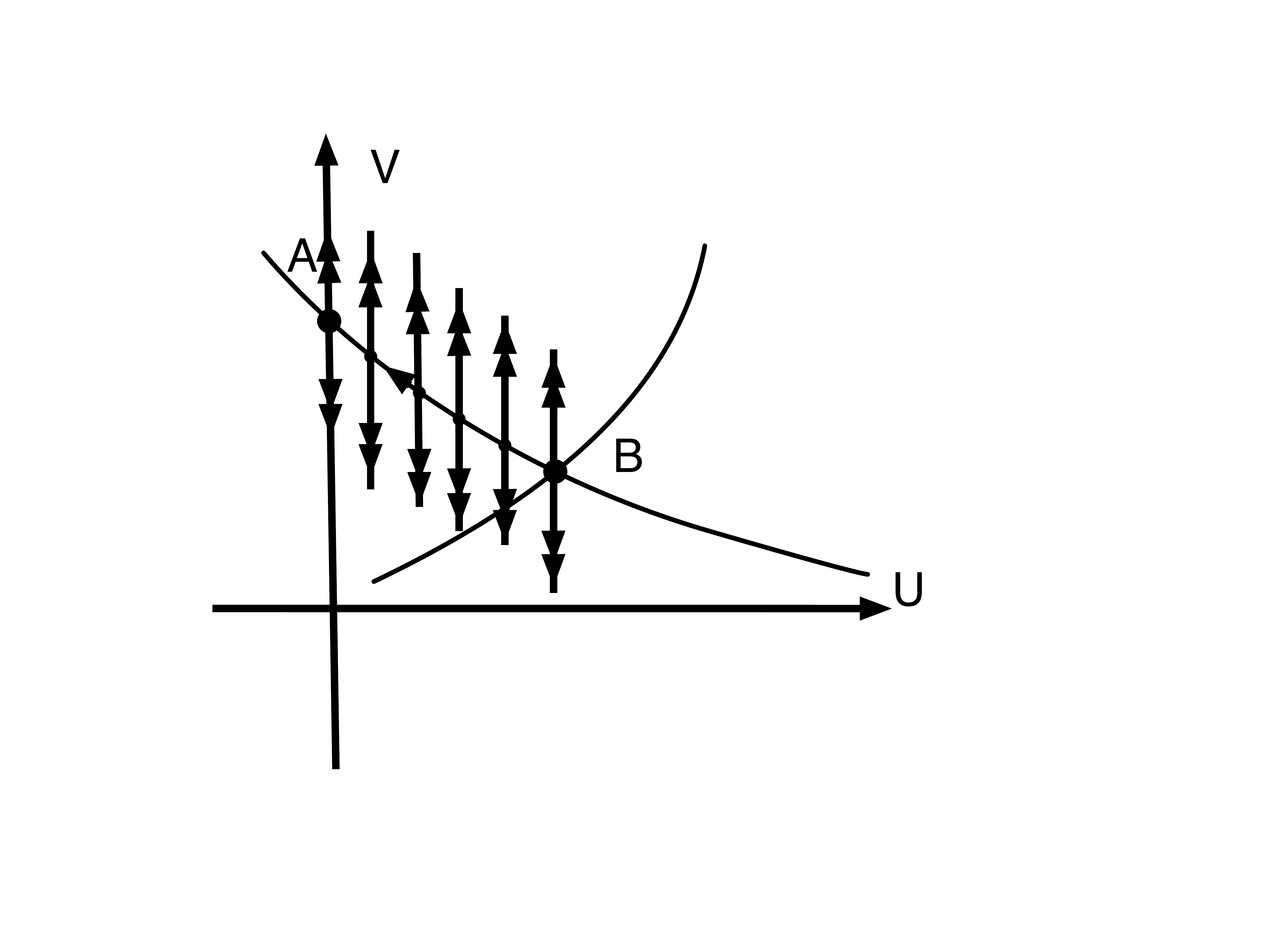} &
\includegraphics[width=3in]{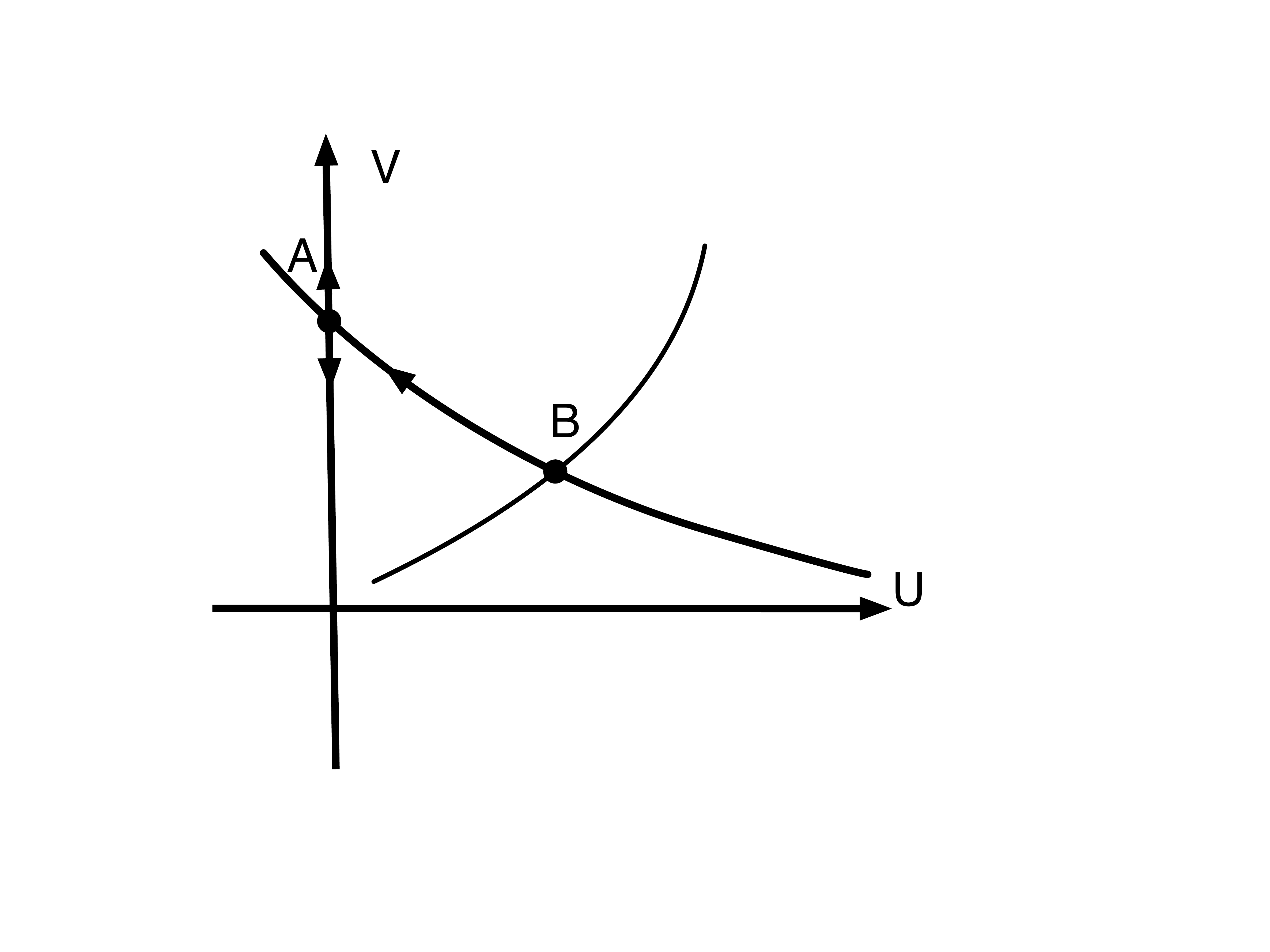} 
\end{array}$
\caption{ The singular limit when:  (left panel) $\omega=0$ and (right panel) $0<\omega\ll1$.}
\label{fig:fronts}
\end{center}
\end{figure}
\end{proof}

In the case of  $\omega \gg1$ the following theorem holds. 

\begin{Theorem} Assume    that $\gamma>2$,  $\mu>0$ and $c>0$ are fixed parameters. Also assume  that $0<\epsilon \ll 1/\omega$. There is  $\omega_0\gg 1$ such that for any $\omega>\omega_0$, there exists  $\epsilon_0 =\epsilon(\omega) >0$ such that for any $\epsilon <\epsilon_0$ in the system \eqref{e:05s}. Equivalently, for the system \eqref{e:05f}, there exists a heteroclinic  orbit connecting $(0,0,1,0)$ and $(\bar{u},0,\bar{v},0).$ Thus, for
 \eqref{e:02pde} there exists a translationally invariant  family of fronts that have the constant states $A=(0,1)$ and $B=(\bar{u},\bar{v})$ as their rest states. \end{Theorem}
\begin{proof}
We denote $\delta=\frac{1}{\omega}$ and rewrite    \eqref{e:10_0}  as follows:
\begin{eqnarray} \label{e:10o}
\left\{\begin{array}{cll}
\delta \frac{du_1}{d\xi}&=&-\frac{1}{c}\left(\frac{\gamma}{1+e^{-\beta(v_1-1)}}u_1(1-u_1)- u_1\right),\\
\frac{dv_1}{d\xi}&=& -\frac{1}{c}(1-(1+u_1)^pv_1). 
\end{array}\right.
\end{eqnarray}   
We then   introduce $z = \xi/ \delta $  and rewrite  \eqref{e:10o}  as:
% then after rescaling (Start from the PDE)
\begin{eqnarray}\label{wd}
\left\{\begin{array}{cll}
\frac{du_1}{d z}&=&- \frac{1}{c} u_1\left(\frac{\gamma}{1+e^{-\beta(v_1-1)}}(1-u_1)- 1\right),\\
\frac{dv_1}{d z}&=& \frac{\delta}{c}(-1+(1+u_1)^pv_1).
\end{array}\right.
\end{eqnarray} 
When  $\delta=0$,  the system \eqref{wd} reads as:
 \begin{eqnarray}\label{wd0} 
 \left\{\begin{array}{cll}
\frac{du_1}{d z}&=&-  \frac{1}{c} u_1\left(\frac{\gamma}{1+e^{-\beta(v_1-1)}}(1-u_1)-1 \right),\\
\frac{dv_1}{d z}&=&0.
\end{array}\right.
\end{eqnarray} 
The slow manifold for this system, which is  also the set of equilibrium points for \eqref{wd0}, consists of two one-dimensional sets: a line $S^1_0=\{(u_1,v_1): u_1=0\}$  and  a curve $$S^2_0=\{(u_1,v_1):\frac{\gamma}{1+e^{-\beta(v_1-1)}}(1-u_1)- 1=0\}.$$   
Linearizing about points from each set, we see that $S^1_0$ is normally  attracting and $S^2_0$ is normally repelling.  Each point of $S^1_0$, including $v_1=1$, has a one-dimensional, linear stable manifold. 
The stable manifold of $S^2_0$  is an open subset of the phase space of  the $(u_1,v_1)$- plane. %It intersects the one-dimensional unstable manifold of  $(0,1)$ along in the open, horizontal  line segment from $(0,1)$ to $(1-\frac{2}{\gamma}, 1)$.

\begin{figure}[t]
\begin{center}$
\begin{array}{cccc}
\includegraphics[width=3in]{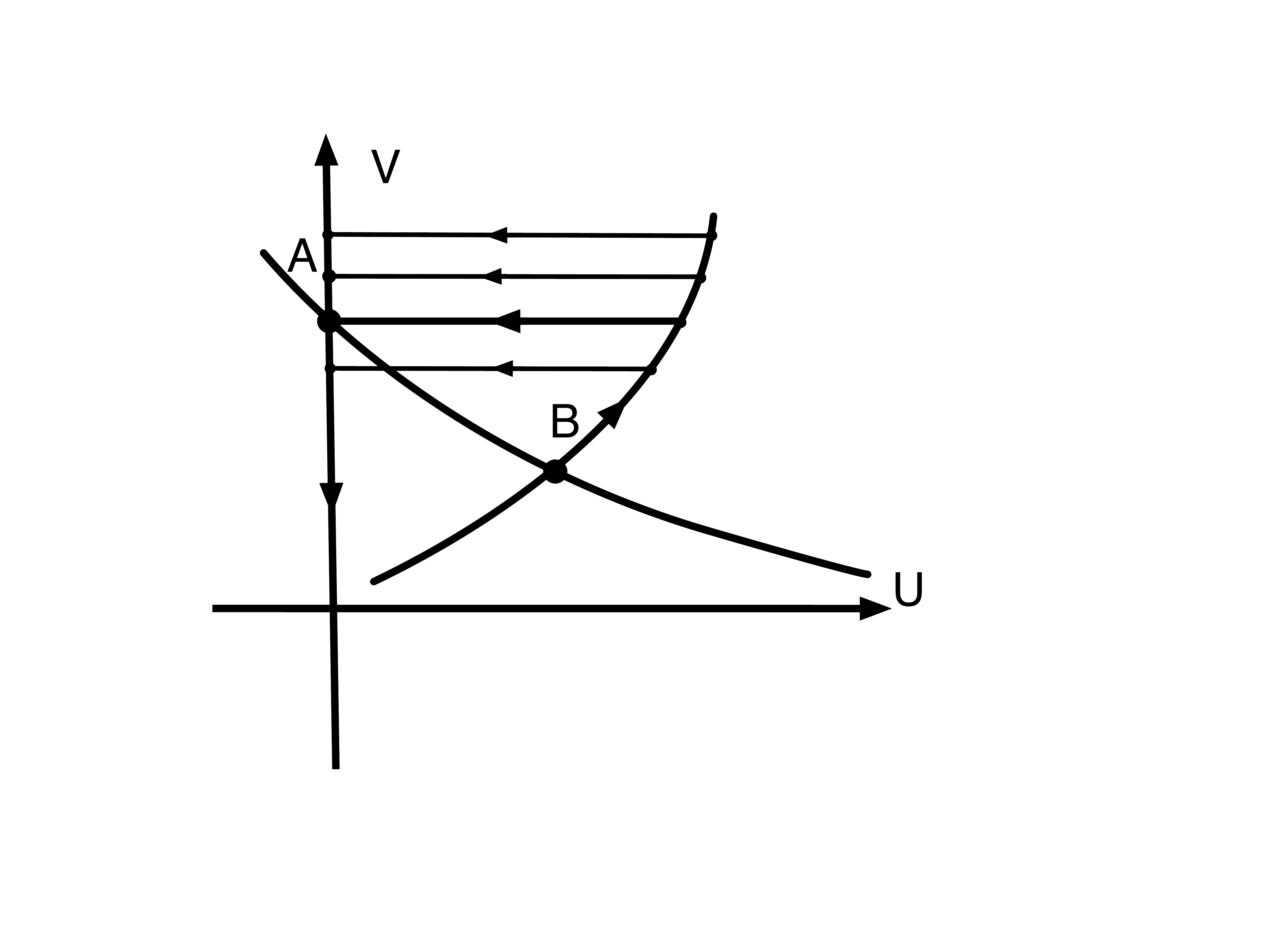} &
\includegraphics[width=3in]{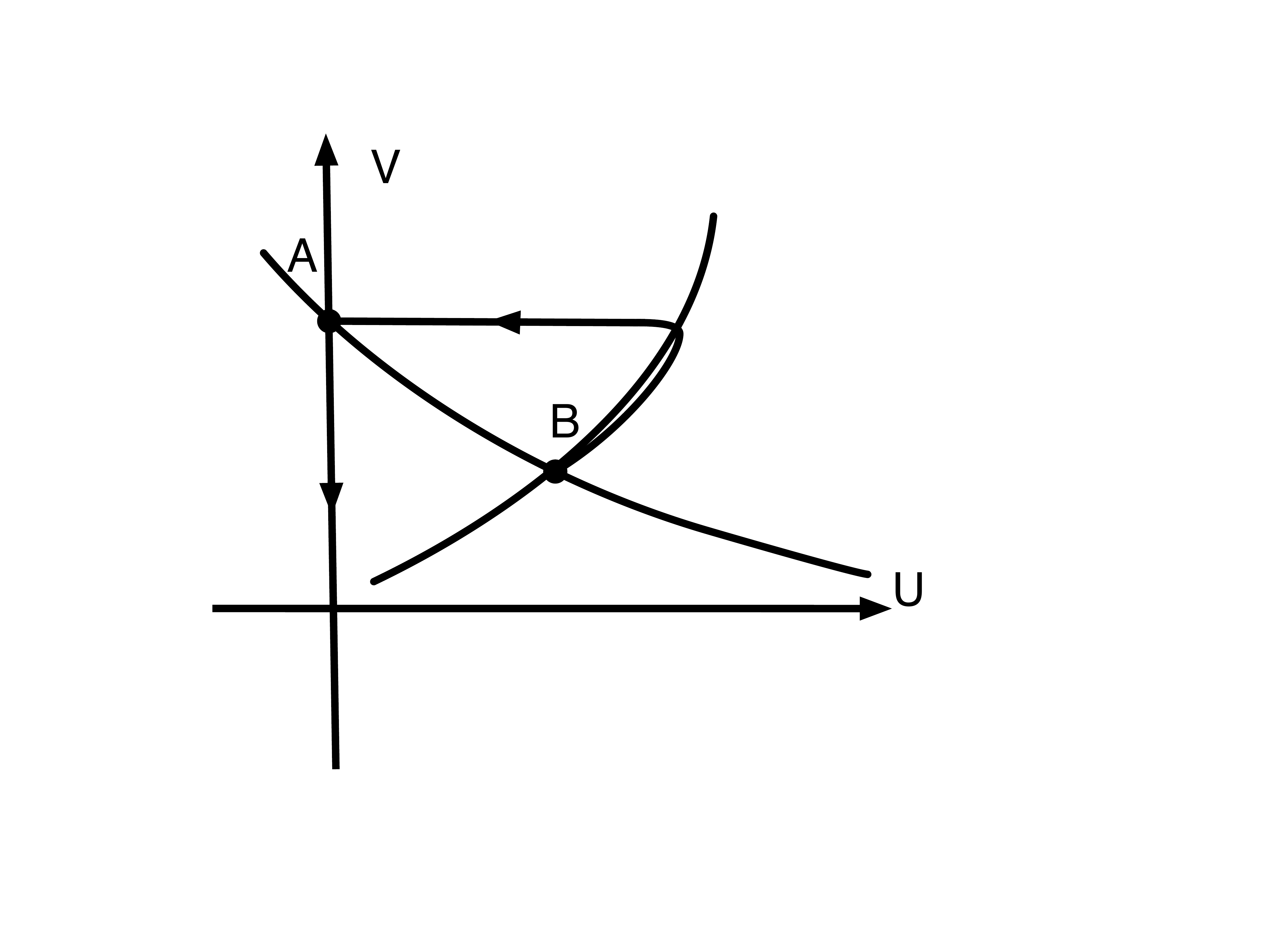} 
\end{array}$
\caption{ A portion of $u_1=0$ is attractive when $\gamma -1-e^{-\beta(v_1-1)}>0$:  the singular limit when $\omega=\infty$ (left panel); the connecting orbit when $ \omega\gg 1$ (right panel).% \textcolor{blue}{{No (a) and (b)}}
}
\label{fig:fronts2}
\end{center}
\end{figure}

The reduced flow on  $S^1_0$ is given by: 
\begin{equation} \label{m1v1}
\frac{dv_1}{d z}= -\frac{1}{c}(1-v_1),\\
\end{equation} 
which has exactly one equilibrium, $v_1=1$, that corresponds to the equilibrium $(0,1)$ in the system \eqref{wd}. Within the set $S^1_0$, this equilibrium of \eqref{m1v1} is repelling.  
For sufficiently small $\delta$, the unstable manifold of  the whole set $S^1_0$ perturbs to the two-dimensional unstable manifold of $(0,1)$.
The reduced system on $S^2_0$  is given by the equation:
\begin{equation}\label{v1z}
\frac{dv_1}{d z}=\frac{\delta}{c} \left(-1+\left(2-\frac{1+e^{-\beta(v_1-1)}}{\gamma}\right)^pv_1\right). 
\end{equation}
 The equation \eqref{v1z} has a single equilibrium at $v_1=\bar v$, which  corresponds to the equilibrium $(v_1,u_1)$ in \eqref{wd}.  Within $S^2_0$, this equilibrium is attracting.
 For sufficiently small $\delta$, the stable manifold of  $v_1=\bar v$   perturbs to the one-dimensional stable manifold of $(\bar v,\bar u)$ in \eqref{wd}.

By the dimension counting, the stable manifold of $S^1_0$  intersects the one-dimensional stable manifold of $(\bar v,\bar u)$ transversally; therefore, for sufficiently small $\delta$ in \eqref{wd}, the unstable manifold of  $(0,1)$ and the stable manifold of $(\bar v,\bar u)$ intersect, thus forming a heteroclinic orbit, 
which is a perturbation of the singular orbit depicted in Figure~\ref{fig:fronts}.

The same argument given in case when $\omega\ll 1$  then shows that this heteroclinic orbit  persists  for  the system \eqref{e:05s} or,  equivalently, \eqref{e:05f}   with sufficiently small values of $\epsilon$. 
\end{proof}

%\textcolor{blue}{{\it Case of  $\omega=O(1)$.???}}
%ROTATION? 
%We first argue  that a perturbation of the singular orbits when $\omega \ll1$ and when $\omega \gg 1$ completely lie in the region above both nullclines given by \eqref{nullclines}. THEY DO NOT. REWRITE. Indeed, when $\omega\ll1$, the   vector field points to  the right when it crosses the  nullcline   $ v_1=\frac{1}{(u_1+1)^p}$  in the region above the nullcline $ u_1=1-\frac{1}{\gamma}(1+e^{-\beta(v_1-1)})$. Similarly,  when $\omega\gg1$, the   vector field points upward when it crosses the  nullcline   $ u_1=1-\frac{1}{\gamma}(1+e^{-\beta(v_1-1)})$ in the region above the nullcline $ v_1=\frac{1}{(u_1+1)^p}$.  This implies that solutions which are small perturbations of the singular solutions and thus are tangent to the unstable never cross the corresponding null-clines and this stay in the region above both of them.
The heteroclinic orbits  in the system \eqref{e:10_0} at intermediate values of  $\omega$  may be traced as continuous deformations of the orbits in singular cases, according to the
theory of rotated vector fields \cite{Perko}. 
We consider the angle between the $u$-axis and the vector given by the right hand
side of  \eqref{e:10_0}: 
$\Phi(u, v) = \tan^{-1}\left( \frac{ f_2(u,v)}{\omega f_1(u, v)}\right).$
It is easy to see that
$$\frac{\partial \Phi}{\partial \omega}=  \frac{- f_1(u,v)  f_2(u, v)} {\omega^2 f_1^2(u,v)+ f_2^2(u,v) }. $$

\begin{figure}[t]
\begin{center}
\includegraphics[width=4in]{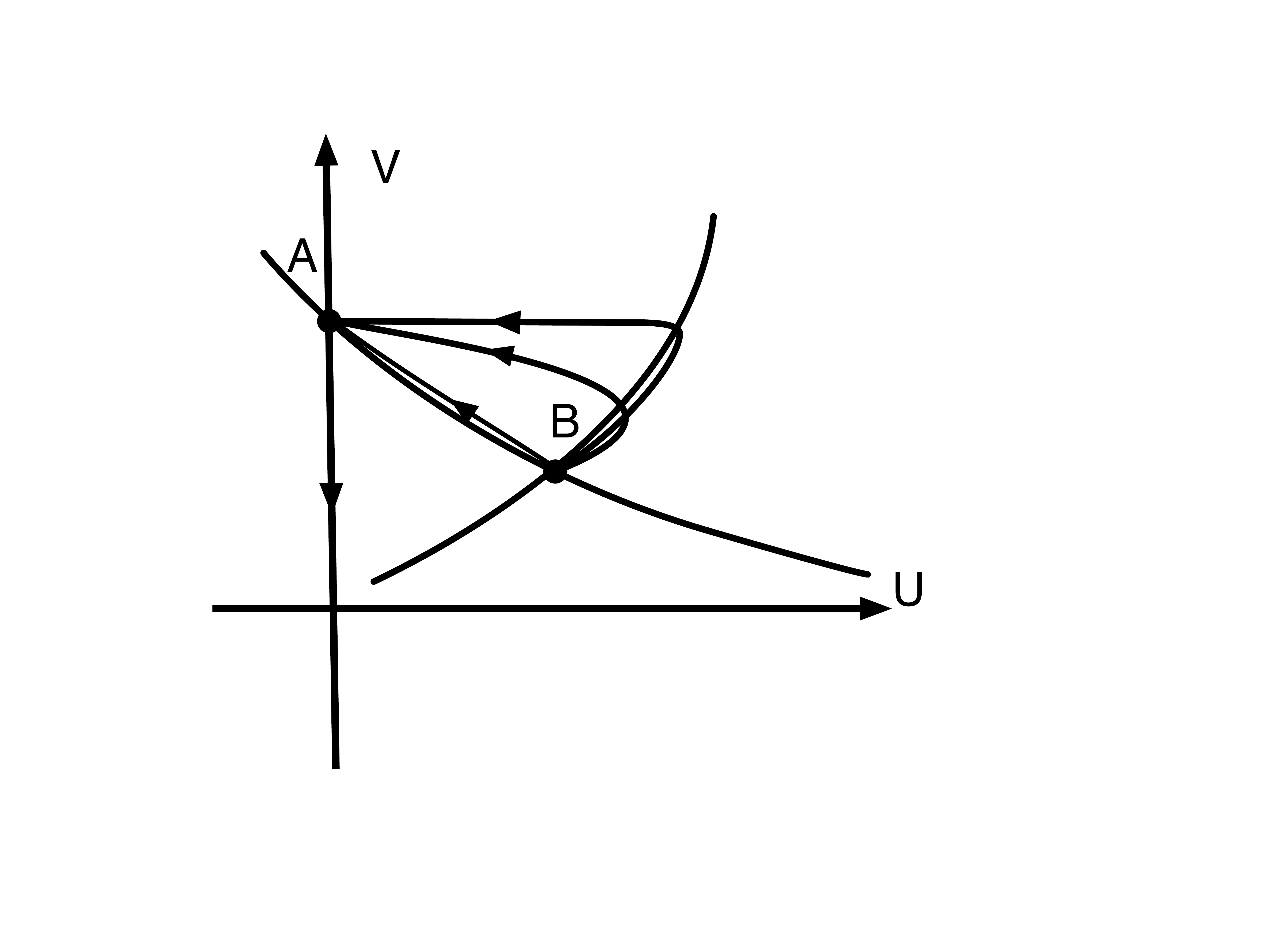} 
\caption{Heteroclinic orbits  for a selection of   $\omega$  values  between zero and infinity.  \label{fig:omegadep}}
\end{center}
\end{figure}

In the region above both nullclines of \eqref{e:10_0}, $f_1<0$ and $f_2 > 0$. Therefore $\frac{\partial \Phi}{\partial \omega} >0$, as $\omega$ decreases  from  infinity to zero, the segment of the
stable manifold $W^s(A)$ of the saddle  $A$ in the described region rotates monotonically  \cite[Section 2]{Perko}, clockwise from its 
limiting position of the singular orbit corresponding to $\omega =\infty$ ($\delta =0$)  to its position of the singular orbit when $\omega =0$.  While   in the region above the both nullclines, $W^s(A)$  for each value of $\omega$ does not cross any of its positions for other values of $\omega$.  
We point out that when $\omega\ll1$, the   vector field points  vertically up  along the  nullcline   $ v_1=\frac{1}{(u_1+1)^p}$  in the region above the nullcline $ u_1=1-\frac{1}{\gamma}(1+e^{-\beta(v_1-1)})$. This implies that the orbits which are small perturbations of the singular orbit  with $\omega =0$  stay above $ v_1=\frac{1}{(u_1+1)^p}$ as they never can cross the this nullcline.  
 On the other hand,    the   vector field allows the orbits to cross the  nullcline   $ u_1=1-\frac{1}{\gamma}(1+e^{-\beta(v_1-1)})$ in the region above the nullcline $ v_1=\frac{1}{(u_1+1)^p}$.  This implies that the orbits for  the intermediate values of $\omega=O(1)$ may be characterized  by the point of intersection of  $W^s(A)$ with $ u_1=1-\frac{1}{\gamma}(1+e^{-\beta(v_1-1)}),$ which moves down the nullcline monotonically.  It follows from \cite[Theorem 2.1]{Yang}  that for any $\omega$   in the system \eqref{e:10_0} there is an orbit that follows $W^s(A)$ and connects the equilibrium $A$ to the equilibrium $B$. The intersection of  one-dimensional stable manifold $W^s(A)$ with  the two-dimensional unstable manifold $W^u(B)$ in the two dimensional phase space is transversal by the dimension counting, therefore will persist  as a solution of the system \eqref{e:05s}, or,  equivalently, the system \eqref{e:05f}  with   sufficiently small  $\epsilon$.

\section{Reduction to the KPP equation \label{KPP}}

In this regime, we consider the PDE system \eqref{e:02stationary_2} under the assumption that $d_1 =O(1)$  and $d_2\ll 1$. To make this more definitive, we set $d_1=1$ and $d_2=\epsilon \ll 1$. %We use the same notation for  new scaled variables as in the previous section understanding that they are not actually the same.  
In  a moving coordinate frame $\xi=x-ct$, system \eqref{e:02stationary_2}  reads as:
\begin{eqnarray}\label{e:021KPP}
\left\{\begin{array}{cll}
u_{\tau}&=& u_{\xi\xi}+cu_\xi+\frac{\gamma}{1+e^{-\beta(v-1)}}u(1-u)- u,\\
v_{\tau}&=&\epsilon v_{\xi\xi}+cv_\xi+\frac{1}{\omega}\left( 1- (1+u)^pv\right).
\end{array}\right. 
\end{eqnarray}

\begin{Theorem} \label{T:3} Assume  that $\epsilon \ll \omega$ in \eqref{e:021KPP}. Also assume that $\gamma >2$ is fixed, and parameters  $\beta$ and $p>0$ are such that
$$\frac{d^2}{d u^2}\left( \frac{ (1-u) u}{1+e^{-\beta(\frac{1}{(1+u)^p}-1)}}\right) <0, \,\, \text{ for }\,\, 0<u<\bar{u}(\gamma,\beta, p), $$
where $\bar{u}$ is the solution of the equation  \eqref{baru}.
For  every fixed value of $ c\geq  \sqrt{2(\gamma -2)}$, there exists $\epsilon_0>0$ such that for any $\epsilon <\epsilon_0$ in \eqref{e:021KPP} there is  $\omega_0=\omega_0(\epsilon)>0$  such that for every $0< \omega <\omega_0$  there exists a translationally invariant  family of fronts in \eqref{e:021KPP}  that have the  equilibria $A=(0,1)$ and $B=(\bar{u},\bar{v})$ as rest states.
As $\epsilon \to 0$ each front converges to a  front in 
\begin{eqnarray}\label{e:021KPPth}
\left\{\begin{array}{cll}
u_{\tau}&=& u_{\xi\xi}+cu_\xi+\frac{\gamma}{1+e^{-\beta(v-1)}}u(1-u)- u,\\
v_{\tau}&=&cv_\xi+\frac{1}{\omega}\left( 1- (1+u)^pv\right).
\end{array}\right. 
\end{eqnarray}
that moves with the same velocity.
\end{Theorem}
\begin{proof}
The proof of this theorem is based on the geometric construction of a heteroclinic orbit in the associated dynamical system, which is corresponding to the front. 
The traveling wave ODE for the system  \eqref{e:021KPP} is: 
\begin{eqnarray}\label{e:02ode}
\left\{\begin{array}{cll}
0&=& u_{\xi\xi}+cu_\xi+\frac{\gamma}{1+e^{-\beta(v-1)}}u(1-u)- u,\\
0&=&\epsilon v_{\xi\xi}+cv_\xi+\frac{1}{\omega}\left( 1- (1+u)^pv\right).
\end{array}\right. 
\end{eqnarray}
We rewrite \eqref{e:02ode}  as a dynamical system:
\begin{eqnarray}\label{e:11KPP}
\left\{\begin{array}{cll}
\frac{du_1}{d\xi}&=&u_2,\\
\frac{du_2}{d\xi}&=&u_1 -c u_2 -\frac{\gamma}{1+e^{-\beta(v_1-1)}}u_1(1-u_1),\\
\frac{dv_1}{d\xi}&=& v_2,\\
\epsilon \frac{dv_2}{d\xi}&=&-cv_2+\frac{1}{\omega} ((1+u_1)^pv_1-1).
\end{array}\right. 
\end{eqnarray}
%where we denote $u_1=u$ and $v_1=v$. 
We also consider an equivalent system that captures the fast dynamics by setting $\zeta=\xi / \epsilon$,
\begin{eqnarray}\label{e:11KPPf}
\left\{\begin{array}{cll}
\frac{du_1}{d\zeta}&=&\epsilon u_2,\\
\frac{du_2}{d\zeta}&=&\epsilon\left(u_1 -c u_2 -\frac{\gamma}{1+e^{-\beta(v_1-1)}}u_1(1-u_1)\right),\\
\frac{dv_1}{d\zeta}&=& \epsilon v_2,\\
\frac{dv_2}{d\zeta}&=&-cv_2+\frac{1}{\omega} ((1+u_1)^pv_1-1).
\end{array}\right. 
\end{eqnarray} 

We study the singular limit of \eqref{e:11KPP}  when $\epsilon \to 0$, thus obtaining  an algebraic description of the slow manifold on which the solution of the limiting system  exists on the following three-dimensional  set:
\begin{equation}\label{m0} M_{\epsilon = 0, \omega}=\left\{(u_1,u_2,v_1,v_2)\vert v_2=\frac{1}{c\omega} ((1+u_1)^pv_1-1)\right\},\end{equation}
with the flow given by:
\begin{eqnarray}\label{e:111KPPlim}
\left\{\begin{array}{cll}
\frac{du_1}{d\xi}&=&u_2,\\
\frac{du_2}{d\xi}&=&u_1 -c u_2 -\frac{\gamma}{1+e^{-\beta(v_1-1)}}u_1(1-u_1),\\
\omega \frac{dv_1}{d\xi}&=&\frac{1}{c}((1+u_1)^pv_1-1),
\end{array}\right.
\end{eqnarray} 
or in a variable $z= \omega\xi$:
\begin{eqnarray}\label{e:111KPPlimz}
\left\{\begin{array}{cll}
\frac{du_1}{dz}&=&\omega u_2,\\
\frac{du_2}{dz}&=&\omega \left(u_1 -c u_2 -\frac{\gamma}{1+e^{-\beta(v_1-1)}}u_1(1-u_1)\right),\\
\frac{dv_1}{dz}&=& \frac{1}{c}((1+u_1)^pv_1-1).
\end{array}\right. 
\end{eqnarray} 
On the other hand,  $M_{\epsilon = 0, \omega}$ is a  set  equilibria for \eqref{e:11KPPf} with  $\epsilon=0:$
\begin{eqnarray}\label{e:11KPPf0}
\left\{\begin{array}{cll}
\frac{du_1}{d\zeta}&=&0,\\
\frac{du_2}{d\zeta}&=&0,\\
\frac{dv_1}{d\zeta}&=&0,\\
\frac{dv_2}{d\zeta}&=&-cv_2+\frac{1}{\omega} ((1+u_1)^pv_1-1).
\end{array}\right. 
\end{eqnarray} 
The linearization  of the system \eqref{e:11KPPf0} about any  point of $M_{\epsilon = 0, \omega}$ has three zero eigenvalues and  a negative eigenvalue $-c$, therefore 
$M_{\epsilon = 0, \omega}$ is normally hyperbolic. 
For sufficiently small $\epsilon$, by Fenichel's invariant manifold theory \cite{Fenichel79},  there exists an invariant, normally attracting manifold $M_{\epsilon, \omega}$  in the  system \eqref{e:11KPPf},
which is an $O(\epsilon)$-order perturbation of  $M_{\epsilon = 0, \omega},$ where:
\begin{equation}\label{me} M_{\epsilon, \omega}=\left\{(u_1,u_2,v_1,v_2)\vert v_2=\frac{1}{c\omega} ((1+u_1)^pv_1-1) +O(\epsilon)\right\}.\end{equation}
The flow generated by \eqref{e:11KPPf} on  $M_{\epsilon, \omega}$ is an $O(\epsilon)$-order perturbation of the flow on $M_0$:
\begin{eqnarray}\label{e:111KPPM0}
\left\{\begin{array}{cll}
\frac{du_1}{d\xi}&=&u_2,\\
\frac{du_2}{d\xi}&=&u_1 -c u_2 -\frac{\gamma}{1+e^{-\beta(v_1-1)}}u_1(1-u_1),\\
\omega \frac{dv_1}{d\xi}&=&\frac{1}{c}((1+u_1)^pv_1-1) +O(\epsilon).
\end{array}\right. 
\end{eqnarray} 

Our further analysis is based on considering another singular  limit in    \eqref{e:111KPPlim} as   $\omega \to 0$. Taking this limit, we obtain a description of a  two-dimensional slow manifold:
\begin{equation} \label{M00} M_{\epsilon=0, \omega=0}=\left\{(u_1, u_2, v_1):\, v_1=\frac{1}{(1+u_1)^p}\right\}\end{equation}
to which the solutions of the limiting system must belong to.
 With $\omega=0$ the system  \eqref{e:111KPPlimz}  reads as:
\begin{eqnarray}\label{e:52}
\left\{\begin{array}{cll}
\frac{du_1}{dz}&=&0,\\
\frac{du_2}{dz}&=&0,\\
\frac{dv_1}{dz}&=& \frac{1}{c}((1+u_1)^pv_1-1).
\end{array}\right. 
\end{eqnarray} 
The  linearization  of \eqref{e:52} about any point $(\tilde u_1, \tilde v_1)$ of the set  $M_{\epsilon=0, \omega=0}$  has two zero eigenvalues and a positive eigenvalue $\frac{1}{c}(1+\tilde u_1)^p$, therefore $M_{\epsilon=0, \omega=0}$ is repelling. 
The dynamics on $M_{\epsilon=0, \omega=0}$ is given by: 
\begin{eqnarray}\label{e:11a}
\left\{\begin{array}{cll}
\frac{du_1}{d\xi}&=&u_2,\\
\frac{du_2}{d\xi}&=&u_1 -c u_2 -\frac{\gamma}{1+e^{-\beta(v_1-1)}}u_1(1-u_1),
\end{array}\right. 
\end{eqnarray} 
or, equivalently by, 
\begin{equation}
0=\frac{d^2u_1}{d\xi^2} +c\frac{du_1}{d\xi}-u_1 +\frac{\gamma}{1+e^{-\beta(\frac{1}{(1+u_1)^p}-1)}}u_1(1-u_1).
\end{equation}
Recall that in the original variables $u_1=u$, so the latter equation  is a traveling wave equation for the scalar partial differential equation:
\begin{equation}\label{e:02}
u_{t}= u_{xx} - u +\frac{\gamma}{1+e^{-\beta(\frac{1}{(1+u)^p}-1)}}u(1-u).
\end{equation} 
The equation \eqref{e:02} is a PDE of a Fisher-KPP type  \cite{Fisher, KPP}, at least, for some parameter regimes. To streamline the current proof, we describe these regimes later in this section.  
%%%%
%%%%%%

The existence of  fronts is well known for the Fisher-KPP equation. In particular, it is proved by a trapping region argument  that for $c\geq 2\sqrt{f^{\prime}(0)} = \sqrt{2(\gamma -2)}$ there is a  heteroclinic orbit that converges to its asymptotic limits  in a monotone way and that is a representation of a monotone front. These heteroclinic  orbits are formed by the intersection of the one-dimensional unstable manifold of the equilibrium at $(\bar u, 0)$ and the two-dimensional stable manifold of the equilibrium $(0,0)$ in the two-dimensional phase space.  By dimension counting this intersection is transversal.

Since the set  $M_{\epsilon=0, \omega=0}$ described in  \eqref{M00} is normally hyperbolic, by Fenichel's theory there is an invariant manifold  of \eqref{e:111KPPlimz} which is an $O(\omega)$-order perturbation $M_{\epsilon=0, \omega}$ of $M_{\epsilon=0, \omega=0}$  which is also normally repelling. The flow on that two-dimensional manifold $M_{\epsilon=0, \omega}$ is an $O(\omega)$-order perturbation of the flow given by \eqref{e:11}.

 In the perturbed system \eqref{e:111KPPlim}, or equivalently \eqref{e:111KPPlimz}, with a sufficiently small $\omega >0$, the  equilibrium $(0,0,1)$ is  a saddle with two-dimensional stable manifold and one-dimensional unstable manifold. To show that, we  linearize  \eqref{e:111KPPlimz} about   the equilibrium $(0,0,1)$:
 \begin{eqnarray}\label{e:111KPPlimzlin}
 \left\{\begin{array}{cll}
\frac{du}{dz}&=&\omega u_1,\\
\frac{du_1}{dz}&=&- \omega\left(\frac{\gamma}{2}-1\right)u -\omega c u_1,\\
\frac{dv}{dz}&=& \frac{p}{c} u +\frac{1}{c}v,
\end{array}\right. 
\end{eqnarray} 
and  calculate the eigenvalues of  the linear operator defined by the right-hand-side  of this system. For $\gamma>2$, it  has two negative eigenvalues 
$\left(-\omega c\pm\sqrt{\omega^2c^2-2\omega^2(\gamma-2)}\right)/2$  and a positive eigenvalue $\frac{1}{c}$.
On the other hand, the eigenvalues of the  linearization of \eqref{e:111KPPlimz} about   the equilibrium $(\bar u, 0, \bar v)$    can be deduced from the slow-fast structure of the system \eqref{e:111KPPlimz}.
Since the slow manifold is normally repelling and this equilibrium on the slow manifold is a saddle, then, for small $\omega$, this equilibrium will have two positive eigenvalues and one negative eigenvalue. 
So the equilibrium $(\bar u, 0, \bar v)$  has a one-dimensional stable manifold and a two-dimensional unstable manifold. 
 
Any solution of \eqref{e:111KPPlimz}  approaching $(1, 0, 0)$ does so while staying on the set $M_{\epsilon=0, \omega}$ since this set is repelling. The solution  that belongs to $M_{\epsilon=0, \omega}$  and  leaves $(\bar u, 0, \bar v)$ must follow the direction within the two-dimensional unstable manifold $W^{u}(\bar u, 0, \bar v)$ that is aligned with \eqref{M00}. Indeed, one of the unstable  eigen-directions of $(\bar u, 0, \bar v)$ is transversal to $M_{\epsilon=0, \omega}$, so the intersection of  $W^{u}(\bar u, 0, \bar v)$ with the set \eqref{M00} is one-dimensional. We further consider the  intersection of this one-dimensional set  with  the two-dimensional stable manifold $W^s(1; 0; 0)$ and notice that it is by dimension counting transversal. 
Thus, for a sufficiently small $\omega >0$, this intersection persists as a transversal intersection and thus,  a heteroclinic orbit for  \eqref{e:111KPPlim}, or equivalently \eqref{e:111KPPlimz}, is formed.

We now recall that the set $M_{\epsilon=0, \omega}$ given by \eqref{m0} is normally hyperbolic  and attracting. The normal hyperbolicity of $M_{\epsilon=0, \omega}$  implies that in the full system 
\eqref{e:11KPP}, there exists an invariant manifold $M_{\epsilon, \omega}$  which is an $O(\epsilon)$-order perturbation of $M_{\epsilon=0, \omega}$  and as such  converges to  $M_{\epsilon=0, \omega}$ in the limit $\epsilon \to 0$ . For sufficiently small $\epsilon,$ it is also normally attracting and the flow generated by \eqref{e:11KPP}  on  $M_{\epsilon, \omega}$  is an $O(\epsilon)$-order perturbation of the limiting flow generated by  the system \eqref{e:111KPPlimz}.  

We claim that there exists a heteroclinic orbit of \eqref{e:11KPP}  that asymptotically connects equilibria $(0, 0,1,0) $ and $(\bar u, 0,\bar v, 0)$ and 
which is an   $O(\epsilon)$-order perturbation  of the heteroclinic  orbit that exists on  $M_{\epsilon=0, \omega}$.  According to  \cite{Fenichel79}, any invariant set for the system  \eqref{e:11KPP}   that is sufficiently
close to $M_{\epsilon=0, \omega}$   is located on $M_{\epsilon, \omega}$. Therefore, both equilibria $(0, 0,1,0) $ and $(\bar u, 0,\bar v, 0)$ belong to $M_{\epsilon, \omega}$. Because $M_{\epsilon=0, \omega}$  is normally attracting, the two-dimensional unstable manifold of $(\bar u, 0,\bar v, 0)$ must stay on  $M_{\epsilon, \omega}$, and so does any orbit that follows that unstable manifold. On the other hand, the intersection of the three-dimensional stable manifold of $(0, 0,1,0) $  with $M_{\epsilon, \omega}$ is two-dimensional.  When $\epsilon=0$, these two  two-dimensional sets intersect transversally  within the three dimensional set, and therefore, the intersection persists when a perturbation with  a sufficiently small $\epsilon$ is introduced.
\end{proof}

We complete the proof of Theorem~\ref{T:3}  by showing that parameter regimes exist  such that the equation \eqref{e:02} is 
%\begin{equation}\label{e:02}
%u_{t}= u_{xx} - u +\frac{\gamma}{1+e^{-\beta(\frac{1}{(1+u)^p}-1)}}u(1-u).
%\end{equation} 
 a PDE of a Fisher-KPP type  \cite{Fisher, KPP} for some parameter regimes.  
%%%%
The Fisher-KPP type equations are PDEs of the form
$$
u_{t}= u_{xx} + f(u), 
$$
where $f$ satisfies the following conditions: there are two equilibrium points  for the equation, say $0$ and $a$, so 
$f(0)=0$, $f(a)=0$, and 
$f^{\prime} (0)>0$, $f^{\prime} (a)<0$, $f^{\prime\prime} (u)<0 $, for $0<u<a$.
 In  the equation \eqref{e:02} we have 
$$f(u)= -u\left(1- \frac{\gamma}{1+e^{-\beta(\frac{1}{(1+u)^p}-1)}}(1-u)\right),$$
so $f(0)=0$, $f(\bar u)=0$,  and 
$$f^{\prime} (0)= -\left(1-\frac{\gamma}{2}\right)>0, \text{ when } \gamma>2.$$ 
 For any $\gamma>0$, $\beta >0$,  since $\frac{\gamma}{1+e^{-\beta(\frac{1}{(1+\bar u)^p}-1)}} = \frac{1}{1-\bar u}$ and  $\bar u <1$,
\begin{equation}
f^\prime(\bar u)%& -\left(1- \frac{\gamma}{1+e^{-\beta(\frac{1}{(1+\bar u)^p}-1)}}(1-\bar u)\right) -u\left( \frac{\gamma}{1+e^{-\beta(\frac{1}{(1+\bar u)^p}-1)}}+\frac{\gamma(1-u)}{1+e^{-\beta(\frac{1}{(1+\bar u)^p}-1)}} \frac{p\beta}{(1+\bar u)^{p+1}}  \right)\\
=-\bar u\left(\frac{1}{1-\bar u} +\frac{p\beta}{(1+\bar u)^{p+1}}\right) <0.
  \end{equation}

 Below we show that there are values of $\beta$ and $p$ such that $f^{\prime\prime} (u)<0 $ for $0<u <\bar u$. To show that,    we introduce, for $\beta  >0$,  a function: $$h(u)= \frac{1}{1+e^{-\beta\left(\frac{1}{(1+u)^p}-1\right)}}.$$ 
The function $h$ is  decreasing since:
$$h^{\prime}(u)= \frac{- 1 }{\left(1+e^{-\beta\left (\frac{1}{(1+u)^p}-1\right)}\right)^2} \frac{\beta p }{(1+u)^{p+1}}e^{-\beta(\frac{1}{(1+u)^p}-1)} <0  $$
 and convex since:
\begin{eqnarray*}h^{\prime\prime}(u)&=& %\frac{2}{\left(1+e^{-\beta\left (\frac{1}{(1+u)^p}-1\right)}\right)^3} \frac{\beta^2 p^2 }{(1+u)^{2p+2}}e^{-2\beta\left (\frac{1}{(1+u)^p}-1\right)}\\&&+  
% \frac{- 1 }{\left(1+e^{-\beta\left (\frac{1}{(1+u)^p}-1\right)}\right)^2} \frac{\beta^2 p^2 }{(1+u)^{2p+2}}e^{-\beta\left (\frac{1}{(1+u)^p}-1\right)}  \\&&+ 
%  \frac{1 }{\left(1+e^{-\beta\left (\frac{1}{(1+u)^p}-1\right)}\right)^2} \frac{\beta p (p+1) }{(1+u)^{p+2}}e^{-\beta\left (\frac{1}{(1+u)^p}-1\right)}\\
 \frac{ e^{-\beta\left (\frac{1}{(1+u)^p}-1\right)}  }{\left(1+e^{-\beta\left (\frac{1}{(1+u)^p}-1\right)}\right)^2}   \frac{\beta p }{(1+u)^{p+2}} \left(\frac{ \beta p }{(1+u)^{p}}\left( \frac{2 e^{-\beta\left (\frac{1}{(1+u)^p}-1\right)} }{1+e^{-\beta\left (\frac{1}{(1+u)^p}-1\right)}}  - 1\right)
%\frac{ \beta p }{(1+u)^{p}} 
+ (p+1)
\right) \\
&& \geq \frac{ e^{-\beta\left (\frac{1}{(1+u)^p}-1\right)}  }{\left(1+e^{-\beta\left (\frac{1}{(1+u)^p}-1\right)}\right)^2}   \frac{\beta p  (p+1)}{(1+u)^{p+2}} 
>0.
  \end{eqnarray*}
Here, we took into account the fact that: \vspace{-6pt}
$$ \frac{1}{2}\leq \frac{e^{-\beta\left (\frac{1}{(1+u)^p}-1\right)} }{1+e^{-\beta\left (\frac{1}{(1+u)^p}-1\right)}}   <1,  \text{ for } u\geq 0.$$ 
%is an increasing  function on the interval  $u\geq 0$. 
We next investigate the convexity of the function $f$.  
We want to find parameter regimes when
$$f^{\prime\prime} (u) = -2 \gamma h(u) +2\gamma(1-2u) h^{\prime}(u) +\gamma u(1-u) h^{\prime\prime}(u) <0. $$
A straightforward calculation of the derivative  and estimates  on some terms show that 
\begin{eqnarray*}
&& \frac{1}{\gamma}\frac{\left(1+e^{-\beta(\frac{1}{(1+u)^p}-1)}\right)^2}{e^{-\beta(\frac{1}{(1+u)^p}-1)}} f^{\prime\prime} (u) = -2  \frac{\left(1+e^{-\beta(\frac{1}{(1+u)^p}-1)}\right)}{e^{-\beta(\frac{1}{(1+u)^p}-1)}} %-2 (1-u -u) \frac{ e^{-\beta\left (\frac{1}{(1+u)^p}-1\right)  }}{\left(1+e^{-\beta\left (\frac{1}{(1+u)^p}-1\right)}\right)} \frac{\beta p }{(1+u)^{p+1}} \\&&\quad +   u(1-u) \frac{ e^{-\beta\left (\frac{1}{(1+u)^p}-1\right)}  }{\left(1+e^{-\beta\left (\frac{1}{(1+u)^p}-1\right)}\right)}   \frac{\beta p }{(1+u)^{p+2}} \left(\frac{2 \beta p }{(1+u)^{p}} \frac{e^{-\beta\left (\frac{1}{(1+u)^p}-1\right)} }{1+e^{-\beta\left (\frac{1}{(1+u)^p}-1\right)}}  +\frac{ - \beta p }{(1+u)^{p}} + (p+1)\right)\\ &&=-2 
\\&&
+ \frac{\beta p }{(1+u)^{p+2}}  \left( - 2 (1-2u)(1+u) +
 u(1-u)\frac{ \beta p }{(1+u)^{p}} \frac{e^{-\beta\left (\frac{1}{(1+u)^p}-1\right)} -1}{1+e^{-\beta\left (\frac{1}{(1+u)^p}-1\right)}}  +  u(1-u)(p+1)
\right) \\
%&& \leq  -2  \frac{\left(1+e^{-\beta(\frac{1}{(1+u)^p}-1)}\right)}{e^{-\beta(\frac{1}{(1+u)^p}-1)}} + \frac{\beta p }{(1+u)^{p+2}}  \left( - 2 (1-2u)(1+u) + u(1-u)\frac{\beta p }{(1+u)^{p}}+  u(1-u)(p+1) \right) \\
&& \leq  -2  \frac{\left(1+e^{-\beta(\frac{1}{(1+u)^p}-1)}\right)}{e^{-\beta(\frac{1}{(1+u)^p}-1)}} + \frac{\beta p }{(1+u)^{p+2}}  \left( - 2 (1-2u)(1+u) +
 u(1-u)(\beta p +p+1)
\right) \\
%&& \leq  -2  \frac{\left(1+e^{-\beta(\frac{1}{(1+u)^p}-1)}\right)}{e^{-\beta(\frac{1}{(1+u)^p}-1)}} + \frac{\beta p }{(1+u)^{p+2}}  \left( - 2 (1-2u)(1+u) + u(1-u)\frac{\beta p }{(1+u)^{p}}+  u(1-u)(p+1)\right) \\
&& \leq 
 %-2 + \beta p  \left( - 2 (1-2u)(1+u) + u(1-u)(\beta p +p+1) \right)\\&& =-2 + \beta p  \left( - 2 (1-2u)(1+u) + u(1-u)((\beta +1)p+1)\right) \\&& =
\beta p (3- (\beta +1)p)u^2  +(3 + (\beta +1)p) u -2(1+\beta p).
  \end{eqnarray*}
We next show that  $p$ and $\beta$  exist such that the  upper bound   obtained above, which  is a quadratic expression  in $u$,  is negative for $u \in (0, \bar {u}) \subset (0,1)$.

First, we observe that if 
$ 3 - (\beta +1)p >0\label{p}$, then 
\begin{equation}\beta p (3- (\beta +1)p)u^2  +(3 + (\beta +1)p) u -2(1+\beta p) <0,\label{neg}\end{equation}
when   $\rho_- < u <\rho_+$,  where
\begin{equation}
\rho_{\pm} = \frac{- (3+(\beta +1)p) \pm \sqrt {(3+(\beta +1)p)^2 +8\beta p (3- (\beta +1)p) (1+\beta p)}}{2 \beta p (3- (\beta +1)p) }. \label{roots}
\end{equation}
%or, more precisely,  when $\rho_- < u <\rho_+$. 
Therefore, we want to guarantee that  $(0, \bar {u}) \subset (0,1) \subset (\rho_- , \rho_+)$. A sufficient condition for this inclusion is: 
\begin{equation}\frac{- (3+(\beta +1)p) + \sqrt {(3+(\beta +1)p)^2 +8\beta p (3- (\beta +1)p )(1+\beta p)}}{2 \beta p (3- (\beta +1)p)}\geq 1,\label{quad}\end{equation}
which leads to the expression:
%$$- (3+(\beta +1)p) + \sqrt {(3+(\beta +1)p)^2 +8\beta p (3- (\beta +1)p )(1+\beta p)} \geq 2 \beta p (3- (\beta +1)p)$$
%$$ \sqrt {(3+(\beta +1)p)^2 +8\beta p (3- (\beta +1)p )(1+\beta p)} \geq 2 \beta p (3- (\beta +1)p) + (3+(\beta +1)p) $$
%$$  {(3+(\beta +1)p)^2 +8\beta p (3- (\beta +1)p )(1+\beta p)} \geq (2 \beta p (3- (\beta +1)p))^2 +(3+(\beta +1)p)^2  +4\beta p (3- (\beta +1)p) (3+(\beta +1)p)$$
%$$  2 (1+\beta p) \geq \beta p(3- (\beta +1)p)   +  (3+(\beta +1)p)$$
%$$  2 (1+\beta p) \geq 3(\beta p+1) - \beta p (\beta +1)p   +  (\beta +1)p$$
%$$0 \geq (\beta p+1) - \beta p (\beta +1)p   +  (\beta +1)p$$
%$$0 \geq 2\beta p+1 - \beta (\beta +1)p^2   +p$$
$$- \beta (\beta +1)p^2 +(2\beta+1) p  +1  \leq 0.$$
We consider this condition as  quadratic  in $p$. Its  roots are given by:
$$ \frac{(2\beta+1) \pm\sqrt{(2\beta+1)^2+4\beta (\beta +1)}}{2 \beta (\beta +1)},$$
 so the inequality \eqref{neg} occurs when \begin{equation}
 p \geq \frac{(2\beta+1)  + \sqrt{(2\beta+1)^2+4\beta (\beta +1)}}{2 \beta (\beta +1)}.
 \end{equation}
 From \eqref{p}, we get  then the following sufficient condition: 
 \begin{equation}\label{regionlower}
\frac{(2\beta+1)  + \sqrt{(2\beta+1)^2+4\beta (\beta +1)}}{2 \beta (\beta +1)} \leq p< \frac{3}{\beta+1}.
 \end{equation}
 The interval  above is not empty  if  $\beta \geq 2$.
 
 %%%%%%%%%%%

When $3 - (\beta +1)p <0\label{p_}$,   the  inequality  \eqref{neg} 
holds for $0<u<\bar u$ when either the quadratic expression  in \eqref{neg} has no roots or when the smallest root is larger than $1$, since the coefficient of $u^2$ is negative and  if the roots \eqref{roots} are real then they are nonnegative.  
The first case occurs when the following holds:
\begin{equation}{(3+(\beta +1)p)^2 +8\beta p (3- (\beta +1)p) (1+\beta p)} <0
\end{equation}
and the second holds when:
\begin{equation} \frac{- (3+(\beta +1)p) - \sqrt {(3+(\beta +1)p)^2 +8\beta p (3- (\beta +1)p) (1+\beta p)}}{2 \beta p (3- (\beta +1)p)} >1.
\end{equation}
These two regions are complementary to each other in the intersection of  
\begin{equation}\label{regionbp-}
p\geq \frac{(2\beta+1)  + \sqrt{(2\beta+1)^2+4\beta (\beta +1)}}{2 \beta (\beta +1)} \quad \text{ and }\quad  p> \frac{3}{\beta +1}.
\end{equation}
Combining this region with the region described in \eqref{regionlower}, we conclude  that the region where   \eqref{neg} 
holds for $0<u\leq 1$ and therefore for  $0<u<\bar u$ is:
\begin{equation}\label{regionpb}
p\geq \frac{(2\beta+1)  + \sqrt{(2\beta+1)^2+4\beta (\beta +1)}}{2 \beta (\beta +1)}.
\end{equation}
In conclusion, we have proved the following  statement.
 \begin{Proposition} For any $\gamma >2$, if  $\beta>0$ and $p>0$  satisfy
\eqref{regionpb}, then
    the equation \eqref{e:02} is  a Fisher-KPP type equation.
 \end{Proposition}
 \begin{figure}[t]
\begin{center}
\includegraphics[width=3.0in]{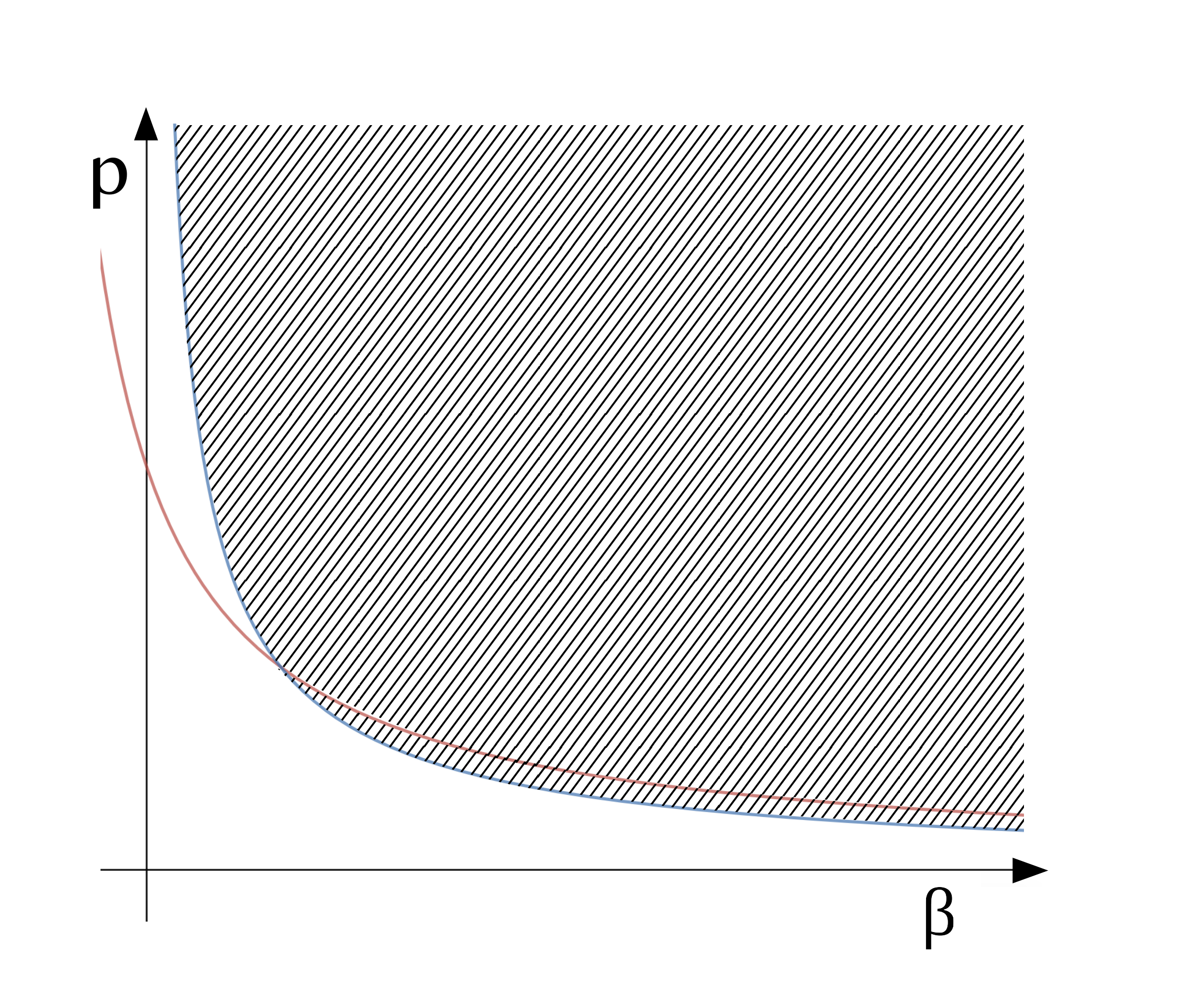} 
\caption{The region in $\beta p$-plane  described by \eqref{regionpb} where the KPP dynamics is guaranteed to be prevalent  in the system  \eqref{e:02ode}, relative to the curve $p=3/(\beta+1)$.}
\end{center}
\end{figure}

On the other hand, we note that since  $f^{\prime\prime}(0)= -\gamma(\beta p+2)/2<0$, then $f^{\prime\prime}(u)<0$ for sufficiently small positive values of $u$.  Moreover, the  solution $\bar u$ of the  equation  \eqref{baru} is a locally increasing  function $\bar u(\cdot)$  of $\gamma$, while  $\bar u(2)=0$. It is easy to see that if  $\bar u$ is sufficiently small,  then $f^{\prime}(u)<0$ for $0<u<\bar u$. Therefore, the following statement holds.
\begin{Proposition} \label{P2}For any fixed $\beta>0$ and $p>0$,  there exists $\gamma_0=\gamma_0(\beta, p)>2$ such that for any $2<\gamma\leq \gamma_0$
 the equation \eqref{e:02} is  a Fisher-KPP type equation.
 \end{Proposition}

 \begin{Remark}
 We point out that the conditions on $p$ and $\beta$ described above are sufficient but not necessary. For any particular $p$ and $\beta$ outside of these intervals, one would have to check if the second derivative of the function 
 $$f(u)= -u\left(1- \frac{\gamma}{1+e^{-\beta(\frac{1}{(1+u)^p}-1)}}(1-u)\right)$$
is negative for $0<u<\bar{u}$. \end{Remark}

 \begin{Remark} The reduction  of the \eqref{e:021KPP} to  the equation \eqref{e:02} holds for $p<0$ as well.  Proposition~\ref{P2} indicates  that as long as $\beta p+1 >0$ and $\gamma$ is close to $2$, the equation  \eqref{e:02} is a Fisher-KPP type equation since the conditions on $f^{\prime}$ and $f^{\prime\prime}$ are satisfied.  Therefore, the Fisher-KPP dynamics are important in the case of $p<0$ as well.
  \end{Remark}

\section{Numerical results \label{NR}}

In this section, we describe some numerical results for the computation of the traveling wave fronts. To compute the front, we 
used the Crank-Nicolson method, which is an implicit finite difference method that is second-order accurate in both time and space. We discretize system (10) on a finite domain $[0 , L]$, with zero Neumann boundary conditions. A decreasing exponential function is used as an  initial condition for $u$ and a constant function is used as an initial condition for $v$.  More precisely, we consider
$$ \begin{cases} 
& u_{\tau} = d_1 u_{xx} +cu_x +f(u,v), \\
& v_{\tau} = d_2 v_{xx}+cv_x +g(u,v), \\
& u_x(0,\tau) = 0 , \;\;\; u_x(L,\tau) = 0 , \\
& v_x(0,\tau) = 0 , \;\;\; v_x(L,\tau) = 0 , \\
&u(x,0) = A e^{-k x}, \;\;\; A>0, k>0 , \\ 
& v(x,0)=B, \;\;\; B>0, 
\end{cases}$$
where $$f(u,v) = \frac{\omega \gamma u(1-u)}{1+ e^{-\beta (v- \alpha)}} -  
\omega u   
\;\;\;\text{and} \;\;\;
g(u,v) = 1 - (1+u)^p v.$$
The discretized scheme of the problem has the following form: 
\begin{equation}  \begin{cases} 
& M_u \mathbf{u}^{l+1} = N_u \mathbf{u}^{l} + 4 
f(\mathbf{u}^{l},\mathbf{v}^{l}) \lambda (\Delta x)^2 , \\
& M_v \mathbf{v}^{l+1} = N_v \mathbf{v}^{l} + 4 
g(\mathbf{u}^{l},\mathbf{v}^{l})  \lambda (\Delta x)^2 , \\ 
\end{cases}\end{equation}
where $l$ represents the number of time steps, $\Delta \tau$ represents the  size  of each time steps, $\mathbf{u}^{l}$ and $\mathbf{v}^{l}$	represent vectors 
of $u$ and $v$ at each point of the domain at time step $l$, 
$$ M_u = 
\begin{pmatrix}
4(1+ \lambda_u) & -\lambda_u(2+c \Delta x) & 0 & ... & 0 \\
-\lambda_u(2-c \Delta x) & 4(1+ \lambda_u) & -\lambda_u(2+c \Delta x) & ... & 0 \\
& & \ddots &  &  &  \\
0 & ... & -\lambda_u(2-c \Delta x) & 4(1+ \lambda_u) & -\lambda_u(2+c \Delta x) \\
0 & ... & 0  & -\lambda_u(2-c \Delta x) & 4(1+ \lambda_u) \\
\end{pmatrix}$$
where $ \lambda_u = d_1  \frac{\Delta \tau}{(\Delta x)^2}$. The matrix $ N_u $ has similar definition.
The matrices $M_v $ and $ N_v$ are  built similarly, but with  $\lambda_v = d_2  
\frac{\Delta \tau}{(\Delta x)^2}$.
We note that the Neumann boundary conditions are incorporated in the matrices.  To 
find the solution, we solve the discretized system for $\mathbf{u}^{l+1}$ 
and  $\mathbf{v}^{l+1}$ at each time step. %An advantage of this method is that we can observe the evolution of the traveling waves  through time. 

Depending on the values of parameters $\gamma$, $\beta$, $p$, and $\omega$, we observed  both monotone and non-monotone fronts .  % \textcolor{red}{I don't see any monotone waves 
%here}.  
The Figures below depict typical shapes of the fronts solution. In these calculations,  $\alpha=1$ and the diffusion constants  are   $d_1 = 0.001$ and $d_2 = 0.002$.

We illustrate small perturbations of the case $\omega=0$  in Figure~\ref{w001} 
%and  ~\ref{w0001},  
where we set $\omega=0.01$. Figure ~\ref{w20} corresponds to a relatively large value  $w=100$. In both cases we set $c=2$.
\begin{figure}[h!]
	\begin{center}
		\includegraphics[scale=0.38]{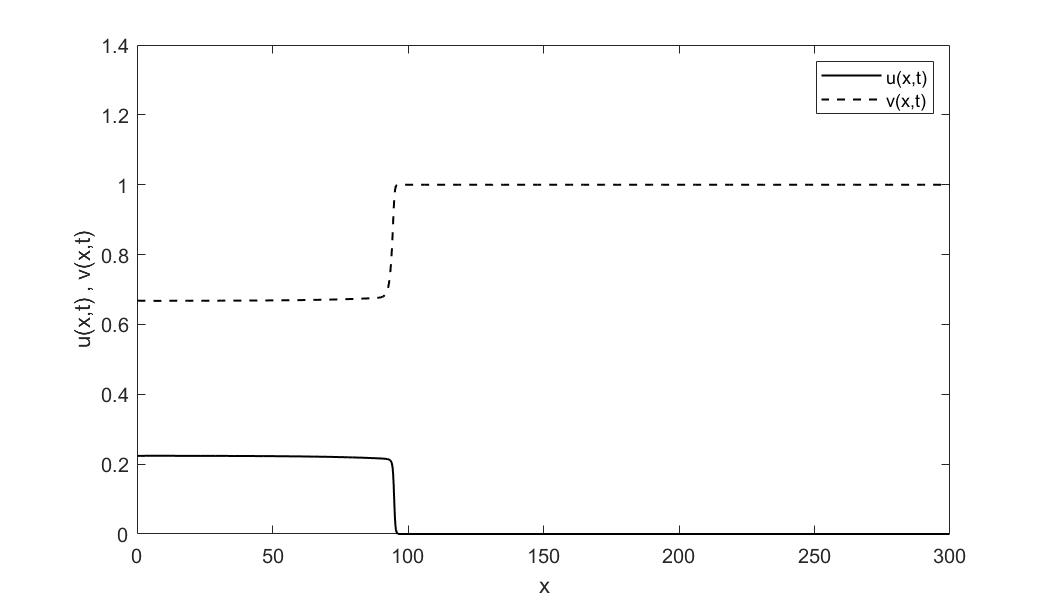}
		\caption{Traveling wave profiles 
			for parameters $\gamma= 1000, \beta = 20, p =2, \omega =0.1, 
			\alpha = 1, c=2.$}
		\label{w001}
	\end{center}
\end{figure}

\begin{figure}[h!]
	\begin{center}
		\includegraphics[scale=0.38]{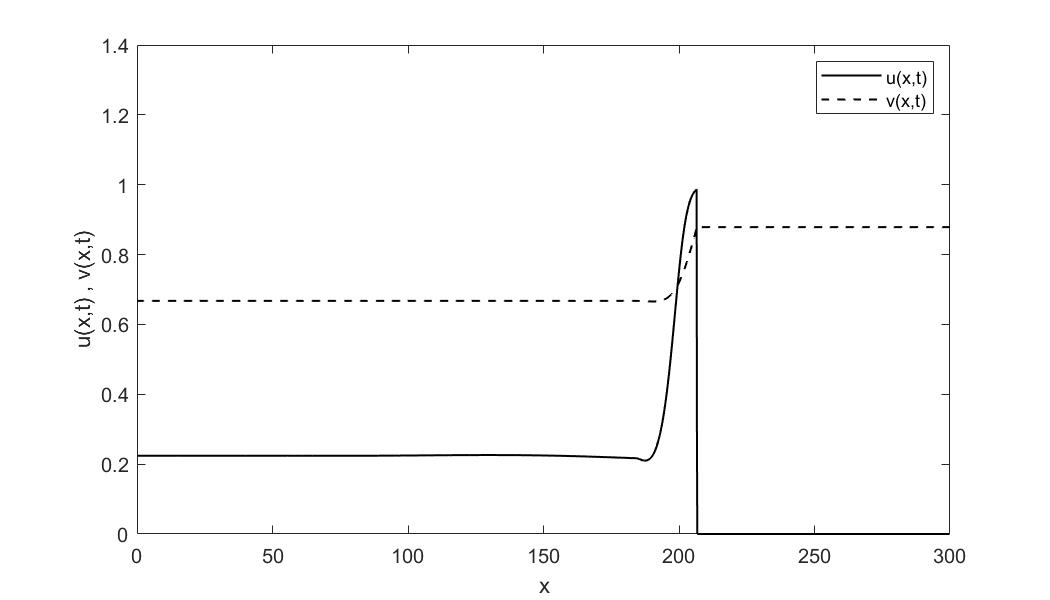}
		\caption{Traveling wave profiles 
			for parameters $\gamma= 5, \beta = 20, p =2, \omega =100, 
			\alpha = 1, c=2.$}
		\label{w20}
	\end{center}
\end{figure}
\FloatBarrier

To illustrate the fronts described in Section~\ref{KPP}, we take $d_u=1$, $d_v=  0.0001$. The simulations produce a  monotone profile  illustrated on Figure \ref{num_front4}.
\begin{figure}[h!]
	\begin{center}
		\includegraphics[scale=0.38]{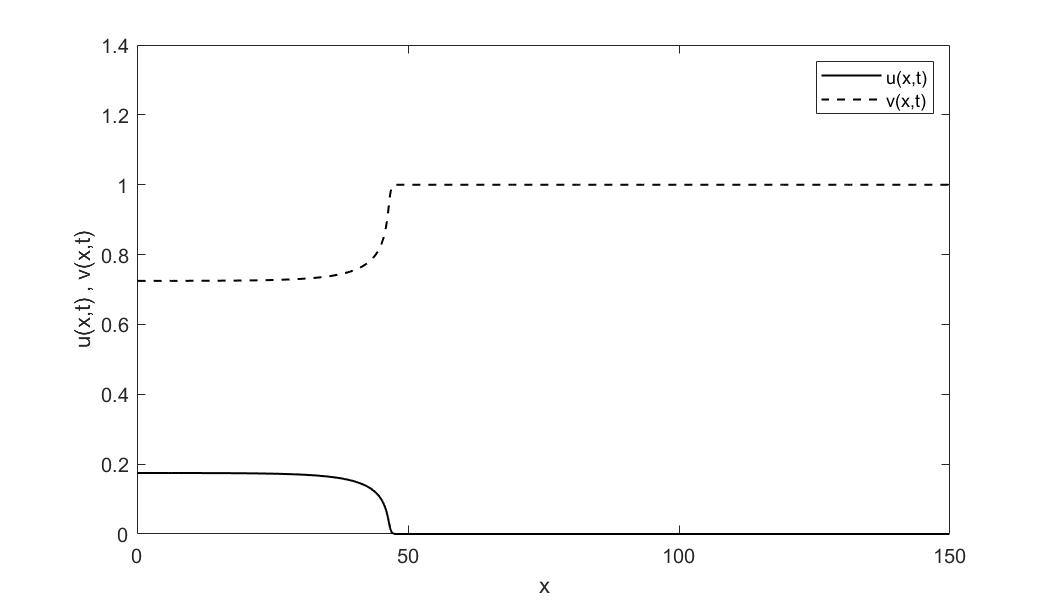}
		\caption{Traveling wave profiles 
			for $\gamma= 300, \beta = 20, p =2, \omega =0.01, c= 25, 
			\alpha = 1.$}
		\label{num_front4}
	\end{center}
\end{figure}
%\FloatBarrier

\section{Acknowledgements and other remarks.}
Ghazaryan was supported by  Faculty Research Grants Program at Miami University. The same grant included a Research Graduate Assistantship  to support  Bakhshi  during the completion of  her  Master Program.  Rodr\'iguez was partially funded by the NSF DMS-1516778.

\end{document}